\documentclass[12pt,oneside,openany, article]{memoir}

\nouppercaseheads

\usepackage[amsthm,thmmarks,hyperref]{ntheorem}

\usepackage[noTeX]{mmap}

\usepackage[english.us]{babel}
\usepackage[leqno]{mathtools}
\usepackage{empheq}

\usepackage{caption}
\usepackage[scr]{rsfso}
\usepackage{upgreek}

\usepackage[compress]{cite}
\usepackage{hypernat} 

\usepackage{fix-cm}
\usepackage[cm]{sfmath}
\usepackage{sansmathaccent}

\usepackage{microtype}
\newtheorem{theorem}{Theorem}[chapter]
\newtheorem{proposition}[theorem]{Proposition}

\newtheorem{corollary}[theorem]{Corollary}

\theoremstyle{remark}
\newtheorem{remark}[theorem]{Remark}

\usepackage{subfig}
\usepackage{wrapfig}
\usepackage{array}

\usepackage{graphicx,xcolor}
\usepackage{eso-pic}
\usepackage[absolute,overlay]{textpos}

\usepackage{pdfpages}
\definecolor{refkey}{rgb}{0,0,1}
\definecolor{labelkey}{rgb}{1,0,0}

\usepackage{todonotes}

\usepackage{xr-hyper}
\usepackage{hyperxmp}
\usepackage{zref, nameref}
\usepackage{url}
\usepackage[pdftex,bookmarks,pdfnewwindow,plainpages=false,unicode,pdfencoding=auto]{hyperref}
\newenvironment{phantomequation}[1][]{\refstepcounter{equation}}{}

\usepackage{bookmark}
\usepackage[verbose]{placeins}
\usepackage{enumitem}
\usepackage{tikz}
\usetikzlibrary{arrows, patterns,
decorations.pathreplacing, decorations.pathmorphing, shapes.symbols, spy}

\usepackage{relsize}

\usepackage{amssymb} 

\setsecheadstyle{\large\bfseries\raggedright}
\setsubsecheadstyle{\normalsize\bfseries\raggedright}

\newenvironment{claim}[1][{\textup{(\theequation)}}]{\refstepcounter{equation}\vglue10pt
\begin{trivlist}
\item[{\hskip\labelsep#1}]}{\vglue10pt\end{trivlist}}
\hypersetup{
colorlinks=true,
linkcolor=black,
citecolor=black,
urlcolor=blue,
pdfauthor={Victor Ivrii},
pdfauthortitle={Professor, Department of Mathematics, University of Toronto},
pdftitle={Microlocal Analysis and Sharp Spectral Asymptotics},
pdfsubject={Sharp Spectral Asymptotics},
pdfkeywords={Microlocal Analysis,  Sharp Spectral Asymptotics},
pdfcopyright={Copyright (C) 2019 Victor Ivrii},
pdfcontactemail={ivrii@math.toronto.edu},
pdfcontacturl={http://www.math.toronto.edu/ivrii},
pdfcontactaddress={40 St. George St.},
pdfcontactcity={Toronto, ON},
pdfcontactpostcode={M5S 2E4},
pdfcontactcountry={Canada},
pdflang={en},
baseurl={http://www.math.toronto.edu/ivrii/monsterbook.pdf/},
pdflicenseurl={http://creativecommons.org/licenses/by-nc-nd/3.0/},
bookmarksdepth={4},
pdfpagelayout=SinglePage
}

\numberwithin{equation}{chapter}

\newcommand{\sC}{\mathscr{C}}

\newcommand{\cL}{\mathcal{L}}
\newcommand{\cF}{\mathcal{F}}
\newcommand{\bR}{\mathbb{R}}
\newcommand{\bC}{\mathbb{C}}
\newcommand{\W}{\mathsf{W}}
\newcommand{\T}{\mathsf{T}}
\newcommand{\w}{\mathsf{w}}
\newcommand{\D}{\mathsf{D}}
\newcommand{\N}{\mathsf{N}}
\newcommand{\const}{\mathsf{const}}
\newcommand{\corr}{\mathsf{corr}}

\renewcommand{\Im}{\operatorname{Im}}

\newcommand{\supp}{\operatorname{supp}}
\newcommand{\dist}{\operatorname{dist}}

\newcommand{\x}{\mathsf{x}}
\newcommand{\y}{\mathsf{y}}

\newcommand{\eff}{{\nu}}

\title{Pointwise Spectral Asymptotics out of the Diagonal near Boundary}

\author{Victor Ivrii\thanks{This research was supported in part by National Science and Engineering  Research Council (Canada) Discovery Grant  RGPIN 13827}}
\begin{document}
\maketitle

\begin{abstract}
We establish uniform (with respect to $x$, $y$) semiclassical asymptotics and estimates for the Schwartz kernel $e_h(x,y, \tau)$ of spectral projector for a second order elliptic operator on the manifold with a boundary. While such asymptotics for  its restriction to the diagonal $e_h(x,x,\tau)$ and, especially, for its trace $\N_h(\tau)= \int e_h(x,x,\tau)\,dx$ are well-known, the out-of-diagonal asymptotics are much less explored, especially uniform ones.

Our main tools: microlocal methods, improved successive approximations and geometric optics methods.

Our results would also lead to \emph{classical} asymptotics of $e_h(x,y,\tau)$ for fixed $h$ (say, $h=1$) and $\tau\to \infty$.
\end{abstract}

\chapter{Introduction}
\label{sect-1}

\paragraph{Asymptotics away from the boundary.}
\label{sect-1.1}
Let us start from the case of asymptotics when points $x$ and  $y$ are disjoint from the boundary. While sharp classical asymptotics of $e (x,x,\tau)$ for elliptic operators are known from 
L.~H\"ormander \cite{hoermander:spectral}, and could be traced to B.~M.~Levitan \cite{levitan:asymp, levitan:spect} and V.~G.~Avakumovi{\v{c}} \cite{avak:eigen}, I could not find asymptotics of $e(x,y,\tau)$, despite it could be easily derived by the method of Fourier integral operators combined with Tauberian theorem.

\begin{theorem}
\label{thm-1.1}
Let $A=A^\w (x,hD,h)$ be a self-adjoint scalar pseudo-differential operator with $\sC^K$-symbol ($K=K(d)$) which is $\xi$-microhyperbolic in $B(0,1)\subset X$  on the energy  level $\tau$\,\footnote{\label{foot-1} Which means that \begin{gather} 
|a(x,\xi)-\tau|+ |\nabla_\xi a(x,\xi) |\ge \epsilon_0
\label{eqn-1.1}
\end{gather}
where $a(x,\xi)\coloneqq A(x,\xi,0)$.}.
Further, let  $\{\xi\colon a(x,\xi)=\tau\}$ be strongly convex. Then 
\begin{gather}
e_h(x,y,\tau)=e_h^\W (x,y,\tau) +O(h^{1-d}) \qquad  \forall x,y\in B(0,\epsilon )
\label{eqn-1.2}\\
\shortintertext{where}
e_h^\W (x,y, \tau)=(2\pi h)^{-d}  \int _{\{a(\frac{1}{2}(x+y), \xi)<\tau\}} e^{ih^{-1}\langle x-y,\xi\rangle}\,d\xi
\label{eqn-1.3}
\end{gather}
is corresponding Weyl expression.
\end{theorem}

We will prove this theorem and discuss its generalizations and also more precise results in Section~\ref{sect-2}.

\paragraph{Asymptotics near the boundary.}
\label{sect-1.2}
Classical asymptotics of $e (x,x,\tau)$ for second order elliptic operators are known from  R.~Seeley \cite{seeley:sharp, seeley:boundary}, and sharp ones from V. Ivrii. The semiclassical version would carry remainder estimate $O(h^{1-d}\eff(x)^{-\frac{1}{2}})$ and $O(h^{1-d})$ respectively while the trivial remainder estimate would be $O(h^{1-d}\eff(x)^{-1})$ where $\eff(x)=\dist(x,\partial X)$. 

Let 
\begin{align*}
\ell^0(x,y)&\coloneqq |x-y|,\\
\ell\ (x,y) &\coloneqq |x-y|+\eff(x)+\eff(y) .
\end{align*}

\begin{theorem}\label{thm-1.2}
Let 
\begin{gather}
A=\sum_{j,k} \bigl(hD_j-V_j(x)\bigr) g^{jk}(x)\bigl(hD_k-V_k(x)\bigr)+V(x),\qquad
g^{jk}=g^{kj}, 
\label{eqn-1.4}
\end{gather}
be a self-adjoint scalar operator  which is elliptic second order differential operator with $\sC^K$-coefficients ($K=K(d,\delta )$), and it  is $\xi$-microhyperbolic in $B(0,1)$  on the energy  level $\tau$, with either Dirichlet or Neumann boundary condition  on $\partial X \cap B(0,1)$, and $\partial X\in \sC^K$\,\footnote{\label{foot-2} More general boundary conditions could be also considered.}. Then
\begin{enumerate}[wide,label=(\roman*),  labelindent=0pt]
\item
For $d\ge 3$ asymptotics \textup{(\ref{eqn-1.2})} holds where now for Dirichlet or Neumann boundary condition 
\begin{gather}
e_h^\W (x,y, \tau)= e_h^{0,\W} (x,y, \tau)\ [\mp] \ e_h^{0,\W}(x,\tilde{y}, \tau)
\label{eqn-1.5}
\end{gather}
respectively, $e_h^{0,\W}(x,y,\tau)$ defined by \textup{(\ref{eqn-1.3})} and $\tilde{y}$ is a reflected point.
\item
For $d=2$ under one of the following assumptions
\begin{enumerate}[label=(\alph*), itemindent=10pt]
\item
$\ell (x,y) \le h^{\frac{1}{3}+\delta}$,
\item
$\ell (x,y) \ge h^{\frac{1}{3}-\delta}$,
\item
$\eff(x) +\eff(y)\ge C_0(\ell^2 + h^{\frac{2}{3}-2\delta})$
\end{enumerate}
where here and below  $\delta>0$ is an arbitrarily small exponent, asymptotics
\begin{gather}
e_h(x,y,\tau)=e_h^\W (x,y,\tau) +e_{h,\corr} (x,y,\tau)+ O(h^{1-d}) \quad  \forall x,y\in B(0,\epsilon )
\label{eqn-1.6}
\end{gather}
holds with  correction term $e_{h,\corr} (x,y,\tau)$ tol be defined by \textup{(\ref{eqn-5.53})} later.
\item
Finally, for $d=2$ if neither of condition (a)--(c) is fulfilled
\begin{gather}
e_h(x,y,\tau)= O(h^{-1-\delta}).
\label{eqn-1.7}
\end{gather}
\end{enumerate}
\end{theorem}

\begin{remark}\label{remark-1.3}
\begin{enumerate}[wide,label=(\roman*),  labelindent=0pt]
\item
According to  \cite{monsterbook}, Theorem~\ref{monsterbook-thm-8-1-6}  asymptotics (\ref{eqn-1.2}) holds for  $x=y$ in the coordinate system such that
$x_1$ is the distance from $x$ to $\partial X$ in the metrics $g^{jk}(\tau - V)^{-1}$. Such precision is needed for $d=2$ only. 

\item
The leading term $e_h^\W(x,y,\tau)$ is $\asymp h^{-\frac{d-1}{2}}\ell ^{-\frac{d+1}{2}}$. In particular it is $O(h^{1-d})$ as 
$\ell \gtrsim h^{\frac{d-1}{d+1}}$.

\item
The correction term (as $d=2$) is $O(h^{-\frac{3}{2}}\ell^{\frac{1}{2}})$ if $\ell \le h^{\frac{1}{2}}$ and $O(h^{-\frac{1}{2}}\ell^{-\frac{3}{2}})$ if 
$h^{\frac{1}{2}}\le \ell \le \epsilon$.
\item
The trivial estimate (in much more general settings) is
\begin{gather}
|e_h(x,y, \tau)-e^\W_h( x,y,\tau)|\le Ch^{1-d} (1+\ell^{-1}(x,y)).
\label{eqn-1.8}
\end{gather}
\end{enumerate}
\end{remark}

\paragraph{Ideas of proofs. I.}
\label{sect-1.3}
Consider $x,y\in B(0,\frac{1}{2})$.  We already know that  under pretty general assumptions (which will be discussed later), in particular, in the frameworks of Theorems~\ref{thm-1.1} and~\ref{thm-1.2},   the following estimate holds:
\begin{gather}
e(x,x, \tau'+h)-e(x,x, \tau')\le Ch^{1-d} \qquad \text{for\ \ } |\tau'-\tau|\le \epsilon_0.
\label{eqn-1.9}\\
\shortintertext{Then}
|e(x,y, \tau)-e(x,y, \tau')|\le Ch^{-d}|\tau -\tau'|+ h^{1-d}
\label{eqn-1.10}\\
\intertext{and therefore due to Tauberian methods}
|e_h(x,y, \tau)-e_{T,h}^\T (x,y, \tau) |\le CT^{-1}h^{1-d}
\label{eqn-1.11}\\
\shortintertext{where}
e_{T,h}^\T(x,y, \tau)= h^{-1}\int_{-\infty}^\tau F_{t\to h^{-1}\tau'} \bigl(\bar{\chi}_T(t)u_h(x,y,t)\bigr) \,d\tau'
\label{eqn-1.12}
\end{gather}
is the \emph{Tauberian expression}, $u_h(x,y,t)$ is the Schwartz kernel of the propagator $e^{ih^{-1}At}$, ${h\le T\le \epsilon _1}$, $\epsilon_1$ 
is a small constant and $\bar{\chi}_T(t)= \bar{\chi}(t/T)$ and $\bar{\chi}\in \sC_0^\infty ([-1,1])$, $\bar{\chi}(t)=1$ on $(-\frac{1}{2},\frac{1}{2})$.

Assume first that $B(0,1)\subset X$. Then for a small constant $T$ we can construct $u_h(x,y,t)$ as an oscillatory integral and we can construct $e_h^\T(x,y,\tau)$ even as $d=2$ or without strong convexity assumption. However the result could be simplified to $e^\W_h(x,y,\tau)$ in the framework of Theorem~\ref{thm-1.1}  (for scalar operators under strong convexity assumption).  This is done in Section~\ref{sect-2}.

\paragraph{Ideas of proofs. II.}
\label{sect-1.4}
Construction of the propagator is much more complicated close to the boundary, because there could be many generalized billiard rays with at least one reflection from $\partial X$,   from $x$ to $y$ on the energy level $\tau$. If
\begin{gather}
\ell(x,y) \coloneqq |x-y|+\eff(x)+\eff(y) \le \epsilon_1
\label{eqn-1.13}\\
\intertext{the length of all such rays is $\asymp \ell(x,y)$. Further, if} 
\eff(x)+\eff(y)\ge \sigma \ell (x,y)
\label{eqn-1.14}
\end{gather}
with $\sigma \ge C_0\ell(x,y)$ there is exactly  one such ray and it has exactly one reflection, and the  reflection  angle $\ge \epsilon_1 (\eff(x)+\eff(y))/\ell (x,y)$. It allows us in the framework of Theorem~\ref{thm-1.2} under assumption (\ref{eqn-1.14}) with $\sigma = h^{\frac{1}{2}-\delta}\ell(x,y)^{-\frac{1}{2}} + C_0\ell (x,y)$
  to construct a reflected wave as an oscillatory integral and simplify it. It will be done in Section~\ref{sect-5}.

So, we should consider a case when (\ref{eqn-1.14}) is violated and we will do it in Section~\ref{sect-4}. As $|x-y|\le Ch^{\frac{1}{2}-\delta}$ the propagator could be constructed by the standard method of successive approximations with unperturbed constant coefficients operator. 

Furthermore, in the framework of Theorem~\ref{thm-1.2} under certain assumptions, which are complementary to (\ref{eqn-1.14}) as $\ell\le h^{\frac{1}{3}+\delta}$, the reflected wave could be constructed also by  the method of the successive approximations, but with unperturbed the operator 
\begin{gather}
\bar{A}(z,hD_z)= h^2D_1^2 + A (0,w',hD_z'),
\label{eqn-1.15}
\end{gather}
where $x=(x_1; x')= (x_1;x_2,\ldots, x_d)$ and  $X\cap B(0,1)= \{x\colon x_1>0\}\cap B(0,1)$ and we make a change of variables $w'=\frac{1}{2}(x'+y')$ and $z'=\frac{1}{2}(x'-y')$, leaving $x_1$ and $y_1$ unchanged.

\paragraph{Ideas of proofs. III.}
\label{sect-1.5}
Finally, to estimate the Tauberian expression rather than to find its asymptotics we apply in Section~\ref{sect-3} only methods of propagations. Let us illustrate this away from the boundary. Let $Q_{1,2}$ be  pseudodifferential operators with the symbols equal $1$  in $\rho$-vicinity of $\{\xi\in \Sigma (w,\tau)\colon \nabla _\xi a(w,\xi)\parallel \bar{x}-\bar{y}\}$; this set consists of just two points. Then 
\begin{gather}
e_h^\T (\bar{x},\bar{y},\tau) = Q_{2x} e_h ^\T(x,y,\tau) \,^t\!Q_{1y} |_{x=\bar{x},y=\bar{y}}+ O(h^s)
\label{eqn-1.16}
\end{gather}
provided microlocal uncertainty principle $\rho\times \rho \ell \ge h^{1-\delta}$ is fulfilled. It allows us to upgrade  $|e^\T_h(x,y,\tau)|\le Ch^{1-d}\ell^{-1}(x,y)$ to
\begin{gather}
|e_h^\T (x,y,\tau)|\le C \rho^{d-1} h^{1-d}\ell^{-1}(x,y).
\label{eqn-1.17}
\end{gather}
 We need also $\rho \ge C_0\ell$.

\chapter{Asymptotics inside domain}
\label{sect-2}

\section{Tauberian asymptotics}
\label{sect-2.1}

As we already mentioned, asymptotics 
\begin{gather}
e_h(x,x,\tau)=\kappa_0(x,\tau)h^{-d}+O(h^{1-d})
\label{eqn-2.1}
\end{gather}
is known under much more general assumptions. It holds for matrix operators under $\xi$-microhyperbolicity condition at point $x$ on the energy level $\tau$ 
and for Schr\"odinger operator without any condition in dimension $d\ge 3$ (see \cite{monsterbook}, Sections~\ref{monsterbook-sect-5-4}  and~\ref{monsterbook-sect-5-3} correspondingly).

Asymptotics (\ref{eqn-2.1}) implies (\ref{eqn-1.9}) and then (\ref{eqn-1.10}). Then the standard Tauberian arguments (see f.e. \cite{monsterbook}, Section~\ref{monsterbook-sect-5-2})  imply (\ref{eqn-1.11})--(\ref{eqn-1.12}) with  small constant $T$. 

\section{Weyl asymptotics}
\label{sect-2.2}

While asymptotics (\ref{eqn-2.1}) holds in much more general assumptions, to construct $u_h(x,y,t)$ we need to impose some restrictions. Let $A$ be a scalar operator\footnote{\label{foot-3} Construction also works for matrix operators with characteristic roots of constant multiplicity.}. While construction of propagator as an oscillatory integral is possible even if $\xi$-microhyperbolicity condition is violated, we assume that it is fulfilled from the beginning, as we need it anyway.

Without any loss of the generality one can assume that 
\begin{gather}
a(y,0)< 0.
\label{eqn-2.2}\\
\intertext{Indeed, otherwise we can achieve it by a corresponding gauge transformation. Then the strong convexity of $\Sigma (y,\tau)=\{\theta\colon a(y,\theta)=\tau\}$ implies that}
\langle \nabla _\theta a(y,\theta),\theta\rangle \ge \epsilon_0\qquad \forall \theta\in \Sigma (y,\tau).
\label{eqn-2.3}
\end{gather}

\begin{proposition}\label{prop-2.1}
Let  $A$ be a scalar operator, $\xi$-microhyperbolic in  $B(\bar{x}, 2\epsilon )\subset X$ on energy level $\tau$ and let $x,y\in B(\bar{x}, \epsilon)$. Let
\textup{(\ref{eqn-2.3})} be fulfilled. Then 
\begin{gather}
u_h(x,y,t)\equiv  \int e^{ih^{-1}\Phi  (x,y,t,\theta)} B (x,y,t,\theta,h)\,d\theta
\label{eqn-2.4}\\
\intertext{modulo functions such that}
F_{t\to h^{-1}\tau'} \bar{\chi}_T(t)v_h(x,y,t) = O(h^s)\qquad \forall x, y\in B(0,\epsilon)
\label{eqn-2.5}\\
\intertext{as $T\le \epsilon_1$, $|\tau'-\tau|\le \epsilon_1$ where}
\Phi (x,y,t,\theta)  =\varphi  (x,y,\theta) +t  a(y,\theta),
\label{eqn-2.6}\\
\intertext{$\varphi  (x,y,\theta)$ satisfies stationary eikonal equation}
a(x, \nabla _x \varphi)= a(y,\theta)
\label{eqn-2.7}\\
\shortintertext{and}
\varphi  (x,y,\theta)|_{\langle x-y,\theta\rangle=0}=0,
\label{eqn-2.8}\\[2pt]
\nabla_x \varphi  (x,y,\theta)|_{x=y}= \theta
\label{eqn-2.9}\\
\shortintertext{and}
B(x,y,t,\theta,h)\sim \sum_{n\ge 0} B_{n} (x,y,t,\theta)h^n,
\label{eqn-2.10}
\end{gather}
and amplitudes $B_{jn}$ are defined from the Cauchy problems for transport equations; all functions are uniformly smooth; here and below $s$ is an arbitrarily large exponent.
\end{proposition}

\begin{proof}
The proof  is original L.~H\"ormander's construction (see M.~Shubin~\cite{shubin:spectral}, Theorem~20.1). 
\end{proof}

While in \cite{hoermander:spectral,shubin:spectral} this decomposition was used to calculate $e^\T_h (x,y,\tau)$ as $x=y$ (actually, it's classical variant), we will use it 
without this restriction. Due to $\xi$-microhyperbolicity 
\begin{multline}
F_{t\to h^{-1}\tau'} \bigl(\bar{\chi}_T(t) u_h(x,y,t)\bigr)\\
\sim (2\pi h)^{-d} h \int_{\Sigma (y,\tau')}  e^{i h^{-1} \varphi (x,y,\theta) } B'(x,y,   \theta)\, d\theta:da(y,\theta)
\label{eqn-2.11}
\end{multline}
with $\Sigma (y,\tau)= \{\theta\colon a(y,\theta)=0\}$, $d\theta:da(y,\theta)$ a natural density on $\Sigma (y,\tau')$ and $B'(x,y,\theta)$ also allowing (\ref{eqn-2.10}) decomposition with uniformly smooth $B'_n$.

\begin{proposition}\label{prop-2.2}
Let in $B(\bar{x},2\epsilon_1)$ both $\xi$-microhyperbolicity and strong convexity conditions are fulfilled on energy level $\tau$. Then for  $|\tau'-\tau|\le \epsilon_1$ and $x\ne y$
$\varphi (x,y,\theta)$ has exactly two stationary points $\theta^*_{\pm}(x,y,\tau')$  on $\Sigma (y,\tau')$,  defined from
\begin{align}
\nabla_\theta \varphi (x,y,\theta)=- t_\pm \nabla_\theta a (y,\theta), && a(y,\theta)=\tau', && \pm t_\pm >0.
\label{eqn-2.12}
\end{align}
These points are non-degenerate and  $t_\pm\asymp |x-y|$ in \textup{(\ref{eqn-2.12})}.

Furthermore,
\begin{gather}
\theta^*_{\pm}(x,y,\tau')=\bar{\theta}_{\pm}(x,y,\tau')+O(|x-y|),
\label{eqn-2.13}\\
\intertext{where $\bar{\theta}_{\pm}(x,y,\tau')$ are defined from}
x-y=- t'_\pm \nabla_\theta a (y,\theta), \qquad a(y,\theta)=\tau', \quad \pm t'_\pm >0
\label{eqn-2.14}
\end{gather}
and also $t'_\pm =t_\pm + O(|x-y|^2)$.
\end{proposition}

\begin{proof}
Proof follows trivially from 
\begin{gather}
\varphi (x,y,\theta)= \langle x-y, \theta\rangle + O(|x-y|^2),
\label{eqn-2.15}
\end{gather}
which follows from (\ref{eqn-2.8}), (\ref{eqn-2.9}).
\end{proof}

\begin{proposition}\label{prop-2.3}
In the framework of Proposition~\ref{prop-2.2} the following asymptotics hold 
\begin{gather}
F_{t\to h^{-1}\tau} \bigl(\bar{\chi}_T(t) u_h(x,y,t)\bigr)\sim (2\pi )^{-d} h^{\frac{1-d}{2}}
\sum_{\varsigma = \pm}  e^{ih^{-1}S_\varsigma (x,y,\tau)} B'_\varsigma (x,y,\tau)
\label{eqn-2.16}\\
\intertext{and with $c\ell\le T\le \epsilon$ }
e^\T_{T,h}(x,y,\tau)\sim (2\pi )^{-d} h^{\frac{1-d}{2}}
\sum_{\varsigma = \pm}  e^{ih^{-1}S_\varsigma (x,y,\tau)} B''_\varsigma (x,y,\tau)
\label{eqn-2.17}\\
\shortintertext{with}
S_\varsigma (x,y,\tau)= \varphi (x,y,\theta ^*_\varsigma(x,y,\tau))
\label{eqn-2.18}
\end{gather}
and $B'_\varsigma$, $B''_\varsigma$   also decomposed into asymptotic series albeit with 
$B'_{\varsigma n} \ell ^{n+\frac{d-1}{2}}$ and $B''_{\varsigma n} \ell ^{n+\frac{d+1}{2}}$uniformly smooth.
\end{proposition}

\begin{proof}
Asymptotics (\ref{eqn-2.17})  follows from Proposition~\ref{prop-2.2} and stationary phase principle with effective semiclassical parameter $\hbar =h\ell^{-1}$, $\ell =|x-y|+h$. 

To prove (\ref{eqn-2.18}) we first rewrite (\ref{eqn-1.12}) as  the sum of two integrals with cut-off functions $\phi _1(\tau')$ and $\phi_2)(\tau')$; 
$\phi_j\in \sC_0^\infty$, $\supp (\phi_1 )\subset (\tau-\epsilon, \tau+\epsilon)$, $\supp(\phi_2)\subset (-\infty, \tau-\epsilon/2)$. In the integral with $\phi_1$ we plug (\ref{eqn-2.4}), observe that there are no stationary points in $\{\theta\colon \tau-\epsilon<a(y,\theta)<\tau\}$ and use Proposition~\ref{prop-2.2} and stationary phase principle with effective semiclassical parameter $\hbar =h\ell^{-1}$. In the integral with $\phi_2$ we simply observe that it has a complete asymptotics because it contains a mollification by $\tau'$ (see, f.e. Proposition~\ref{prop-3.5} below).
\end{proof}

\begin{proposition}\label{prop-2.4}
In the framework of Proposition~\ref{prop-2.2} 
\begin{enumerate}[wide,label=(\roman*),  labelindent=0pt]
\item
The following estimates hold:
\begin{gather}
|F_{t\to h^{-1}\tau} \Bigl(\bar{\chi}_T (t)u_h(x,y,t)\Bigr) |\le Ch^{\frac{1-d}{2}}\ell^{\frac{1-d}{2}}
\label{eqn-2.19}
\intertext{and as $c\ell\le T\le \epsilon$ }
|e^\T_{T,h} (x,y,\tau)|\le Ch^{\frac{1-d}{2}}\ell^{\frac{-1-d}{2}}.
\label{eqn-2.20}
\end{gather}
\item
The following estimate holds as $\max(c\ell,\,h^{1-\delta})\le T\le \epsilon$:
\begin{gather}
|e^\T_{T,h} (x,y,\tau)-e^\W_h(x,y,\tau) |\le  Ch^{1-d}.
\label{eqn-2.21}
\end{gather}
\end{enumerate}
\end{proposition}

\begin{proof}
\begin{enumerate}[wide,label=(\alph*),  labelindent=0pt]
\item
Estimates (\ref{eqn-2.19})--(\ref{eqn-2.20}) follow from (\ref{eqn-2.16})--(\ref{eqn-2.17}).
\item
To prove Statement (ii)  for $d\ge 3$ one can observe that both $|e^\T_{T,h} (x,y,\tau|$ and $|e^W(x,y,\tau)|$ do not exceed $Ch^{1-d}$ as $h^{\frac{1}{2}}\le \ell\le \epsilon$, 
while for $\ell\le h^{\frac{1}{2}}$, due to  (\ref{eqn-2.15}), replacing $\Phi (x,y,t,\theta)$ by 
\begin{gather}
\bar{\Phi}(x,y,\theta)=\langle x-y,\theta\rangle  + t a(\frac{1}{2}(x+y),\theta)
\label{eqn-2.22}
\end{gather}
brings an error not exceeding
the right-hand expression of (\ref{eqn-2.20}) multiplied by $C\ell^2h^{-1}$, that is $Ch^{-\frac{1+d}{2}}\ell ^{\frac{3-d}{2}}\le Ch^{1-d}$.

Moreover, if we also replace $B(x,y,t,\theta,h)$ by $1$ then the error will not exceed the right-hand expression of (\ref{eqn-2.20}) multiplied by $C\ell $, that  is
$Ch^{\frac{1-d}{2}}\ell ^{\frac{1-d}{2}}\le Ch^{1-d}$. 

Finally, with this choice of the phase function $\bar{\Phi}(x,y,t,\theta)$ and amplitude $B=1$,  we get $e^\T_{T,h}(x,y,\tau) =e^\W(x,y,\tau) + O(h^\infty)$.

\item
As $d=2$ we need more subtle arguments (which would also work for $d\ge 3$). Indeed, estimate (\ref{eqn-2.20}) implies that  $ |e^\T_{T,h} (x,y,\tau)|\le Ch^{-1}$ as 
$h^{\frac{1}{3}}\le \ell\le \epsilon_1$ and  estimate (\ref{eqn-2.21}) leads to $O(h^{-1})$ error only as $\ell \asymp h$. 

We refer to Proposition~\ref{prop-5.13}(i).
\end{enumerate}
\end{proof}

Finally, Theorem~\ref{thm-1.1}  follows from estimate (\ref{eqn-1.11}) and Proposition~\ref{prop-2.4}.

\section{Improvements and generalizations}
\label{sect-2.3}

\paragraph{Two-term asymptotics.}
We know that  if $A$ is a scalar operator on the closed manifold, $\xi$-microhyperbolic at point $x$ and if \emph{non-looping condition}
\begin{gather}
\upmu_{x,\tau} \{ \xi \in \Sigma (x,\tau) \colon \exists t\ne 0, \uppi _x e^{tH_a (x,\xi)} =x\}=0
\label{eqn-2.23}\\
\intertext{is fulfilled at point $x$ then two-term asymptotics}
e_h(x,x,\tau)=\kappa_0(x,\tau)h^{-d}+\kappa_1(x,\tau)h^{1-d} + o(h^{1-d})
\label{eqn-2.24}
\end{gather}
holds; here $\Sigma (x,\tau)= \{\xi\colon a(x,\xi)=\tau\}$, $ \upmu_{x,\tau}$   is a natural measure on $\Sigma (x,\tau)$, corresponding to a density $dxd\xi :d_\xi a(x,\xi)$, $H_{a}$ is a Hamiltonian field generated by $a$ and $e^{tH_a}$ is a Hamiltonian flow (see, f.e. \cite{monsterbook}, Section~\ref{monsterbook-sect-5-3}). Can we get a two-term  asymptotics for $e_h(x,y,\tau)$? 
\enlargethispage{\baselineskip}

First of all, we need to improve Tauberian estimate (\ref{eqn-1.11}). Asymptotics (\ref{eqn-2.24}) at point $x$ implies
\begin{gather}
|e(x,x, \tau+T^{-1} h)-e(x,x, \tau')|\le CT^{-1}h^{1-d} +o_T(h^{1-d})
\label{eqn-2.25}
\end{gather}
and then
\begin{multline*}
|e(x,y, \tau )-e(x,y, \tau')|\\
\le C|\tau -\tau'| h^{-d}+C|\tau -\tau'|^{\frac{1}{2}}h^{\frac{1}{2}-d} + CT^{-\frac{1}{2}}h^{1-d}+o_T(h^{1-d})
\end{multline*}
with arbitrarily large $T$. Here condition (\ref{eqn-2.23}) is fulfilled at one of  points $x$, $y$. Then the standard Tauberian arguments imply
\begin{gather*}
|e_h(x,y, \tau)-e_{T,h}^\T (x,y, \tau) |\le CT^{-\frac{1}{2}}h^{1-d} + o_T(h^{1-d})
\end{gather*}
and propagation results combined with non-looping condition at one of points $x$, $y$ imply that in the left hand expression one can replace $T$ by  $T'=T(x,y)$ while still having arbitrary  $T$ in the right hand-expression, and then we arrive to
\begin{gather}
e_h(x,y, \tau) = e_{T,h}^\T (x,y, \tau) +o(h^{1-d})\qquad\text{with\ \ } T=T(x,y).
\label{eqn-2.26}
\end{gather}
Further, if we add another condition
\begin{gather}
\upmu_{x,\tau} \{ \xi \in \Sigma (x,\tau) \colon \exists t\ne 0, \uppi _x e^{tH_a (x,\xi)} =y\}=0
\label{eqn-2.27}
\end{gather}
or a similar condition, obtained by permutation of $x$ and $y$ (let's call it $(\ref{eqn-2.26})'$), then from (\ref{eqn-2.26}) and propagation results we obtain that
\begin{enumerate}[label =(\alph*), wide]
\item 
If $\ell (x,y)\ge \epsilon$ then $e_h(x,y,\tau)=o(h^{1-d})$,
\item
If $\ell(x,y)\le \epsilon$, then we can take $T(x,y)=c\ell (x,y)$ in (\ref{eqn-2.26}).
\end{enumerate} 

\paragraph{Weyl asymptotics.}
In particular, in the framework of Theorem~\ref{thm-1.1}, under extra assumptions (\ref{eqn-2.23}) and either (\ref{eqn-2.26}) or
$(\ref{eqn-2.26})'$, we can apply the machinery developed in Subsection~\ref{sect-2.2} (with the only exception $d=2$ and neither $\ell \gg h^{\frac{1}{3}}$ nor $\ell \ll h^{\frac{1}{3}}$). One can prove easily that then 
\begin{gather}
\bigl(e(x,y,\tau)=e^\W (x,y,\tau)\bigr)+o(h^{1-d}) \qquad  \forall x,y\in B(0,\epsilon )
\label{eqn-2.28}
\end{gather}
where 
$e_h^\W (x,y, \tau)$  is defined by (\ref{eqn-1.3}) with $a$ replaced by $a+ha_1$ where $a_1$ is the subprincipal symbol of $A$, that is with the main term defined by (\ref{eqn-1.3}) and with the second term
\begin{gather}
-(2\pi h)^{1-d}\int _{\Sigma (z, \tau)} a_{1}(z,\theta) e^{i\langle x-y,\theta\rangle}\,d\theta :d_{\theta} a(z,\theta),\qquad z=\frac{1}{2}(x+y).
\label{eqn-2.29}
\end{gather}

\paragraph{Generalizations.}
\begin{enumerate}[wide,label=(\roman*),  labelindent=0pt]
\item
We can consider the case when $a(x,\xi)$ has characteristic roots $\lambda_j(x,\xi)$ of constant multiplicity; and we need only to assume that it happens 
as $|\lambda_j (x,\xi)-\tau|\le \epsilon$.
\item
If we consider only $y\colon y-x \in \Gamma$ where $\Gamma$ is a cone in $\bR^d$ with a vertex at $0$, then in virtue of propagation results (see \cite{monsterbook}, Section~\ref{monsterbook-sect-2-2}) we need the previous assumption only for $\xi$ such that
\begin{gather}
\Gamma \cap (K (x,\xi) \cup -K(x,\xi))\ne \emptyset
\label{eqn-2.30}
\end{gather}
where $K(x,\xi)$ is a cone dual to $K'(x,\xi)$ which is a connected component of 
\begin{gather*}
\{\eta \colon ((\eta \cdot \partial_\xi a(x,\xi))v,v)\ge 0\ ; \forall v\colon \|(a-\tau) v\|\le \epsilon \|v\|\}.
\end{gather*}
\item
Also, instead of  strong convexity we need to assume only that at points $\xi$ satisfying (\ref{eqn-2.30}), the matrix of the curvatures of the surface
$\{\xi \colon \lambda (x,\xi)=\tau\}$ has at least two eigenvalues, disjoint from $0$. Then it would be similar to the case $d\ge 3$ in Proposition~\ref{prop-2.4}. Probably, it would suffice to have just one eigenvalue, to recover more delicate arguments of the case $d=2$ in Proposition~\ref{prop-2.4}. 
\end{enumerate}

\chapter{Microlocal methods}
\label{sect-3}

\section{Propagation}
\label{sect-3.1}
Recall that we consider Schr\"odinger operator (\ref{eqn-1.4}) and for such operator under $\xi$-microhyperbolicity near boundary condition on the energy level $\tau$ instead of on-diagonal asymptotics (\ref{eqn-2.1}) asymptotics with a boundary-layer type term and remainder estimate $O(h^{1-d})$ holds; see \cite{monsterbook}, Section~\ref{monsterbook-sect-8-1}. In particular, under Dirichlet or Neumann boundary conditions (and we consider only those for simplicity) asymptotics (\ref{eqn-1.5}) holds as $x=y$ and also (\ref{eqn-1.9}) holds. Then (\ref{eqn-1.10})--(\ref{eqn-1.12}) hold and all we need is to rewrite  Tauberian expression (\ref{eqn-1.12}) in more explicit terms.

Let us   decompose
\begin{gather}
u_h(x,y,t)= u_h^0 (x,y,t) + u_h^1 (x,y,t)
\label{eqn-3.1}\\
\intertext{where $u_h^0 (x,y,t)$ is the solution for the ``free space'' and $u_h^1(x,y,t)$ is a reflected wave, satisfying}
hD_t u_h^1 =A(x,hD,h)u_h ^1,
\label{eqn-3.2}\\
B(x,hD,h) u_h^1 |_{\partial X} = -B(x,hD,h) u_h^0 |_{\partial X},
\label{eqn-3.3}
\end{gather}
where $B(x,hD,h)$ is a boundary operator. Then
\begin{gather}
e^\T_h(x,y,\tau)= e^{0,\T}_h(x,y,\tau)+ e^{1,\T}_h(x,y,\tau),
\label{eqn-3.4}
\end{gather}
where $e^{0,\T}_h(x,y,\tau)$ and $e^{1,\T}_h(x,y,\tau)$ are derived by Tauberian expression  (\ref{eqn-1.12}) from  $u_h^0 (x,y,t) $ and $u^1 (x,y,t)$ correspondingly
with $T=\epsilon_1$.

While $u^0_h(x,y,t)$ can be constructed as an oscillatory integral, our purpose is to construct $u^1_h(x,y,t)$  in some way.
To do this, in this section  we apply the improved method of successive approximations near boundary, in the same way as the standard method was applied in \cite{monsterbook}, Section~\ref{monsterbook-sect-7-2} while in the next Section~\ref{sect-5} we use the geometric optics method. 

Recall that now $A$ is a second order elliptic operator and
$X\cap B(0,1)= \{x\colon \eff(x)>0\}\cap B(0,1)$ with $|\nabla \eff|\ge \epsilon _0$. Without any loss of the generality one can assume that 

\begin{claim}\label{eqn-3.5}
$X\cap B(0,1)=\{x\colon x_1>0\}\cap B(0,1)$ with  $V_1(x)=0$,  $g^{1k}=\updelta_{1k}$ 
\end{claim}
and therefore 
\begin{gather}
\eff(x)= x_1,\qquad \ell(x,y)=|x'-y'|+x_1+y_1
\label{eqn-3.6}
\end{gather}
and  $A$ is $\xi$-microhyperbolic on the energy level $\tau$ at $x$ as long as
\begin{gather}
|V(x)-\tau|\ge \epsilon_0,
\label{eqn-3.7}
\end{gather}
and in this case $\{\xi\colon a(x,\xi)=\tau\}$ is strongly convex.

\begin{proposition}\label{prop-3.1}
\begin{enumerate}[wide,label=(\roman*),  labelindent=0pt]
\item
As $T\le \epsilon \ell(x,y)$ the following estimate holds:
\begin{gather}
|F_{t\to h^{-1}\tau}\Bigl(\bar{\chi}_T(t)u^1_h(x,y,t)\Bigr)|\le Ch^{-d}\Bigl(\frac{h}{T}\Bigr)^s.
\label{eqn-3.8}
\end{gather}
\item
Let condition \textup{(\ref{eqn-3.1})} be fulfilled. Then for $\max(C\ell(x,y),\, h^{1-\delta})\le T\le T'\le \epsilon _1$   the following estimate holds:
\begin{gather}
|F_{t\to h^{-1}\tau}\Bigl(\bigl(\bar{\chi}_{T'}(t)- \bar{\chi}_{T} (t)\bigr) u^1_h(x,y,t)\Bigr)|\le Ch^{-d}\Bigl(\frac{h}{T}\Bigr)^s.
\label{eqn-3.9}
\end{gather}
\end{enumerate}
\end{proposition}

\begin{proof}
Statement (i)  follows from the finite speed of propagation and Statement (ii) follows from the fact that the speed of propagation with respect to $x$ is disjoint from $0$
(either along $\partial X$, if $|\xi'|$ is disjoint from $0$ or in direction of $x_1$ if $\xi_1$ is disjoint from $0$)--see \cite{monsterbook}, Chapter~\ref{monsterbook-sect-3}.
\end{proof}

\begin{corollary}\label{coro-3.2}
Let $e^{1,\T}_h(x,y,\tau)$ be defined by \textup{(\ref{eqn-1.12})} with $u^1_h(x,y,t)$ instead of $u_h(x,y,t)$.
\begin{enumerate}[wide,label=(\roman*),  labelindent=0pt]
\item
Then in $e^{1,\T}_h(x,y,\tau)$ one can replace $T=\epsilon_1$ by $T= C\max(\ell(x,y),\, h^{1-\delta})$.
\item
Further, in the framework of Proposition~\ref{prop-3.1}(ii) and $\ell (x,y)\ge h^{1-\delta}$ one can replace $\bar{\chi}_{T}(t)$ with $T=\epsilon_1$
by $\bigl(\bar{\chi}_{T}(t)-\bar{\chi}_{T'}(t)\bigr)$ with $T=C_0\ell(x,y)$ and $T'=C_0^{-1}\ell(x,y)$.
\end{enumerate}
\end{corollary}

We need to be more precise about propagation especially with respect to $x_1$ and $\xi_1$ as $\ell ^0\asymp \ell $. 
Let us $\bar{x}$, $\bar{y}$ and $\bar{w}$ be fixed final values of $x$, $y$ and $w$ respectively and we need to consider case $\ell \coloneqq \ell(\bar{x},\bar{y})\ge h^{1-\delta}$.

\begin{proposition}\label{prop-3.3}
Let $h^{1-\delta}\le T \le \epsilon_1$, $T=C_0\ell$ and
\begin{phantomequation}\label{eqn-3.10}\end{phantomequation}
\begin{gather}
\rho T \ge h^{1-\delta}, \qquad \rho \ge T^2.
\tag*{$\textup{(\ref*{eqn-3.10})}_{1,2} $}\label{eqn-3.10-*}
\end{gather}
Let   $b(\bar{y}',\bar{\xi}')=\tau$ and 
\begin{gather}
|\nabla _{y'} b(y', \bar{\xi}  )_{y'=\bar{y'}}|\le \rho T^{-1}
\label{eqn-3.11}\\
\intertext{where here and below}
b(x',\xi')\coloneqq a(x,\xi)|_{x_1=\xi_1=0};
\label{eqn-3.12}
\end{gather}
and let $Q_1(y,\eta')$ have a symbol supported in $(T,   \rho)$-vicinity of $(\bar{y},\bar{\xi}')$. 

\begin{enumerate}[wide,label=(\roman*),  labelindent=0pt]
\item
Let  $Q_2(x,\xi')$ have a symbol, equal $0$ in   $\{(x, \xi')\colon |\xi'-\bar{\xi}'|\le C_1\rho\}$. Then
\begin{gather}
F_{t\to h^{-1}\tau} \Bigl(\bar{\chi} _T(t) Q_{2x} u^1 _h(x,y,t) \,^t\!Q_{1y}\Bigr)=O(h^s).
\label{eqn-3.13}
\end{gather}
\item
Assume instead that  $Q_2(x,hD')$ has a symbol, equal $0$ as  $x_1\le C_1\sigma  T$ where here and below
\begin{gather}
\sigma = (\rho ^{\frac{1}{2}} +T). 
\label{eqn-3.14}
\end{gather}
Then \textup{(\ref{eqn-3.13})} holds.

\item
Assume instead that  $Q_2(x,hD')$ has a symbol, equal $0$  in 
\begin{multline}
\{(x',\xi')\colon |x'-\bar{y}'  + t \nabla _{\xi'}b(\bar{y}', \xi')|_{\xi=\bar{\xi}'}|\le C_1\rho T\}\\ 
\forall t\colon \epsilon T \le \pm t \le C_0T.
\label{eqn-3.15}
\end{multline}
Assume also that
\begin{gather}
\rho^2 T\ge h^{1-\delta},\qquad \rho \ge C_0T.
\label{eqn-3.16}\\
\shortintertext{Then}
F_{t\to h^{-1}\tau} \Bigl(\chi^\pm  _T(t) Q_{2x} u^1 _h(x,y,t) \,^t\!Q_{1y}\Bigr)=O(h^s).
\label{eqn-3.17}
\end{gather}
where $\chi^\pm \in \sC^\infty_0(\bR)$ is supported in $\{t\colon \epsilon \le \pm t\le 1\}$.
\end{enumerate}
\end{proposition}

\begin{proof}
\begin{enumerate}[wide,label=(\roman*),  labelindent=0pt]
\item
Observe that under assumptions $\textup{(\ref{eqn-3.10})}_{2}$ and (\ref{eqn-3.11}) the propagation speed with respect  to $\xi'$ does not exceed 
$C_1 (\rho +T)$ in 
\begin{gather*}
\{(x,\xi')\colon |x'-\bar{y}'|+|x_1|+|y_1|\le C_1T,\ |\xi'-\bar{\xi}' |\le C_1\rho\}.
\end{gather*}
This implies Statement~(i) by the standard methods of propagation; see Subsection~\ref{monsterbook-sect-5-1-2} of \cite{monsterbook}. Here  $\textup{(\ref*{eqn-3.10})}_{1}$  is the microlocal uncertainty principle).

\item
Due to Statement~(i) the propagation is confined to $\{(x,\xi')\colon |b(x',\xi')-\tau |\le C_0\sigma^2\}$ and thus to $|\xi_1|\le C\sigma $ and then to
$x_1 \le  C_1\sigma \ell$. The microlocal uncertainty principle with respect to $(x_1,\xi_1)$, that is
$\sigma \times\sigma \ell\ge h^{1-\delta}$ is fulfilled due to $\textup{(\ref{eqn-3.10})}_{1} $ and therefore we can consider  propagation with respect to $(x_1,\xi_1)$ as 
$x_1\ge \sigma T$.

\item
Due to Statement~(i) the propagation in the time direction $\pm t>0$ is confined to the domain (\ref{eqn-3.15}) and here (\ref{eqn-3.16}) is a microlocal uncertainty principle. 
\end{enumerate}
\end{proof}

\begin{remark}\label{remark-3.4}
\begin{enumerate}[wide,label=(\roman*),  labelindent=0pt]
\item
We can achieve (\ref{eqn-3.11})  by  gauge transformation, not affecting $x_1,\xi_1$. Further, if
\begin{gather}
\rho \ge T
\label{eqn-3.18}
\end{gather}
then (\ref{eqn-3.11})   is fulfilled automatically.

\item
Obviously, condition  (\ref{eqn-3.11})  is important only as $\rho \le T$, which is not the case as $T\le h^{\frac{1}{2}}$. Therefore we can cover also $\ell \le h^{1-\delta}$ with $T=h^{1-\delta}$.

\item
Both statements of Proposition~\ref{prop-3.3} hold for $\tau'\colon |\tau-\tau'|\le \sigma^2 =(\rho +\ell^2)$.
\end{enumerate}
\end{remark}

\section{Spectral estimates}
\label{sect-3.2}

The following proposition will be useful in the proof of Theorem~\ref{thm-1.2}:

\begin{proposition}\label{prop-3.5}
\begin{enumerate}[wide,label=(\roman*),  labelindent=0pt]
\item
Let $L \ell^0 (x,y)\ge h^{1-\delta}$. Let $\phi \in \sC_0^\infty (\bR)$. Then
\begin{gather}
\int \phi ((\tau-\tau')/L ) F_{t\to h^{-1}\tau'} \Bigl((\bar{\chi}_T (t) u_h(x,y,t)\Bigr)\,d\tau'=O(h^s).
\label{eqn-3.19}
\end{gather}
\item
Let $L \ell (x,y)  \ge h^{1-\delta}$. Let $\phi \in \sC_0^\infty (\bR)$. Then
\begin{gather}
\int \phi ((\tau-\tau')/L ) F_{t\to h^{-1}\tau'} \Bigl((\bar{\chi}_T (t) u^1 _h(x,y,t)\Bigr)\,d\tau'=O(h^s).
\label{eqn-3.20}
\end{gather}
\end{enumerate}
\end{proposition}

\begin{proof}
The left-hand expressions are just $\phi ((hD_t-\tau) /L) u_h(x,y,t)\bigr|_{t=0}$ and $\phi ((hD_t-\tau) /L) u^1_h(x,y,t)\bigr|_{t=0}$ correspondingly and we apply finite speed of propagation.
\end{proof}

\begin{proposition}\label{prop-3.6}
Under $\xi$-microhyperbolicity condition on the energy level $\tau$ estimates  \textup{(\ref{eqn-1.8})} and
\begin{gather}
|e_h(x,y,\tau)|\le Ch^{1-d} \bigl(1+\ell^{0\,-1}(x,y)\bigr)
\label{eqn-3.21}
\end{gather}
hold.
\end{proposition}

\begin{proof}
In virtue of Subsection~\ref{sect-2.1} for $h\le \T\le \epsilon$
\begin{multline*}
e_h(x,y,\tau) = e^{\T}_h(x,y,\tau) +O(T^{-1}h^{1-d}) \\
=e^{0,\T}_h(x,y,\tau) +e^{1,\T}_h(x,y,\tau) +O(T^{-1}h^{1-d}) 
\end{multline*}
and for $T=\epsilon_1 \ell(x,y)$
\begin{gather*}
e^{0,\T}_h(x,y,\tau) =e^{0,\W}(x,y,\tau)+O(h^{-d}(h\ell^{-1})^s)\\
\shortintertext{and}
e^{1,\T}_h(x,y,\tau) =O(h^s)
\end{gather*}
due to the finite speed of propagation, which implies (\ref{eqn-1.8}). Estimate  (\ref{eqn-3.21}) is proven in the same way.
\end{proof}

As $h^{1-\delta}\le \ell(\bar{x},\bar(y))\le \epsilon_1$ let us introduce $\bar{\xi}^{\prime\pm}$:
\begin{gather}
\nabla _{\xi'} b(\bar{w}', \xi')|_{\xi'=\bar{\xi}^{\prime\pm}}= t(\bar{y}-\bar{x}), \qquad  b(\bar{w}', \bar{\xi}^{\prime\pm})=\tau,\qquad \pm t>0
\label{eqn-3.22}
\end{gather}
with $w'= \frac{1}{2}(\bar{x}'+\bar{y}')$. Due to strong convexity $\bar{\xi}^{\prime\pm}$ are defined uniquely and $t \asymp \ell(\bar{x},\bar(y))$.

\begin{proposition}\label{prop-3.7}
Let $h^{1-\delta}\le T \le \epsilon_1$, $T=C_0\ell$,  \textup{(\ref{eqn-3.18})} and \textup{(\ref{eqn-3.16})} be fulfilled. Then
\begin{gather}
 e_h^{1,\T } (\bar{x},\bar{y},\tau) = Q_{2x}  e_h^{1,\T} (x,y,\tau) \,^t\!Q_{1y} |_{x=\bar{x},y=\bar{y}} +O(h^s)
 \label{eqn-3.23}
\end{gather}
where $Q_1$, $Q_2$ are operators with symbols equal $1$ in $(C_1\ell, C_1\rho)$ vicinities of $(\bar{w}', \bar{\xi}^{\prime\pm})$.
\end{proposition}

\begin{proof}
Due to Proposition~\ref{prop-3.5}
\begin{multline}
e_h^{1,\T } (x,y,\tau) \\
= h^{-1}\int_{-\infty}^\tau \bar{\phi } ((\tau -\tau')L^{-1}) F_{t\to h^{-1}\tau'} \bigl(\bar{\chi} _T(t)u^1 _h(x,y,t)\bigr) \,d\tau' +O(h^s)
\label{eqn-3.24}
\end{multline}
where $\bar{\phi } \in \sC^\infty (\bR)$, $\bar{\phi } (\tau)=1$ as $\tau \le 1$ and $\bar{\phi } (\tau)=0$ as $\tau \ge 2$, $L\ell \ge h^{1-\delta}$. 

Due to Proposition~\ref{prop-3.1}(ii) we can replace $\bar{\chi}_T(t)$ by $\chi_T(t)= \bar{\chi}_T(t)- \bar{\chi}_{\epsilon T}(t)$.
Then Proposition~\ref{prop-3.3}(iii)  implies that if $\psi _1, \psi_2\in \sC_0^\infty $ are supported in $\epsilon\ell$-vicinities of $\bar{y}$ and $\bar{x}$ respectively and
$Q'=Q'(hD')$ is $0$ in $C\rho$-vicinities of $\bar{\xi}^{\prime\pm}$ then as $|\tau'-\tau |\le L = h^{1-\delta}L^{-1}$
\begin{gather*}
F_{t\to h^{-1}\tau'} \Bigl(\chi  _T(t) Q'_x\psi_2(x) u^1_h(x,y,t)\psi_1(y)\Bigr)=O(h^s), \\
F_{t\to h^{-1}\tau'} \Bigl(\chi _T(t) \psi_2(x) u^1_h(x,y,t)\psi_1(y)\,^t\!Q'_y\Bigr)=O(h^s)
\end{gather*}
Indeed condition now are stronger than those of Proposition~\ref{prop-3.3}(iii).

Then these equalities and (\ref{eqn-3.24}) imply (\ref{eqn-3.23}).
\end{proof}

\begin{proposition}\label{prop-3.8}
In the framework of Proposition~\ref{prop-3.7}
\begin{gather}
|e^{1,\T}_{T,h}(x,y,\tau)|\le Ch^{1-d} \ell^{-1}\sigma \rho^{d-2}
\label{eqn-3.25}
\end{gather}
with $\sigma = h^{\frac{1}{2}-\delta'}\ell^{-\frac{1}{2}}+\ell$.
\end{proposition}

\begin{proof}
Due to Propositions~\ref{prop-3.1}(ii) and~\ref{prop-3.7} we need to estimate the right-hand expression of (\ref{eqn-3.24}) with symbols of $Q_1$ and $Q_2$
supported in $\Omega$, the union of $(C_1 \ell, C_1 \rho)$-vicinities of $(\bar{w}',\bar{\xi}^{\prime+})$ and $(\bar{w}',\bar{\xi}^{\prime-})$ intersected with
\begin{gather*}
\{\xi \colon |b(x',\xi')-\tau|\le C\sigma^2\},
\end{gather*}
 and  with $\bar{\chi}_T(t)$ 
replaced by  $\chi_T(t)= \bar{\chi}_T(t)- \bar{\chi}_{\epsilon T}(t)$. Observe that modulo $O(h^s)$ 
\begin{multline*}
e_h^{1,\T} (x,y,\tau)  \equiv h^{-1} \int_{-\infty}^\tau F_{t\to h^{-1}\tau'} \Bigl (\chi_T(t) u^1_h (x,y,t) \Bigr)\,d\tau'\\
= T^{-1}  F_{t\to h^{-1}\tau} \Bigl (\beta_T(t) u^1_h (x,y,t) \Bigr)
\end{multline*}
with $\beta(t)=t^{-1}\chi(t)$. 

Since $\ell \asymp \ell^0$ it sufficient to prove it for $e^{\T}_{T,h}(x,y,\tau)$.  Expressing $u_h(x,y,t)= \int e^{ih^{-1}t\tau} \,d_\tau e_{h}(x,y,\tau)$ we see that expression in question equals
\begin{gather*}
\int \hat{\beta}\bigl (\frac{(\tau-\tau')T}{h}\bigr) \,d_{\tau'} Q_{2x}e_h(x,y,\tau')\,^t\!Q_{1y}
\end{gather*}
which does not exceed
\begin{align*}
C\sup _{\tau'\colon |\tau'-\tau|\le L}& |Q_{2x} e_h(x,y,\tau'+ hT^{-1},\tau')\,^t\! Q_{1y}| +C'h^s.\\
\intertext{Since $e_h(x,y,\tau)$ is the Schwartz kernel of the orthogonal projector this does not exceed}
C\sup _{\tau'\colon |\tau'-\tau|\le L} & |Q_{2x} e_h(x,z,\tau'+ h T^{-1} ,\tau')\,^t\! Q_{2z}|_{z=x} |^{\frac{1}{2}}\\ 
\times& |Q_{1z} e_h(z,y,\tau'+ hT^{-1},\tau')\,^t\! Q_{1y}|_{z=y}|^{\frac{1}{2}} +C'h^s.
\end{align*}
Applying Tauberian estimate  for $x=y$ we see that the right-hand  expression does not exceed 
\begin{gather*}
Ch^{1-d}\ell^{-1} \sup _{\tau'\colon |\tau'-\tau|\le L} \int _{\Sigma (x,\tau')\cap (\Omega  \times \bR_{\xi_1})} \, d\xi:d_\xi a .
\end{gather*}
The integral in the right-hand expression does not exceed $C\sigma \rho^{d-2}$. Finally, recall that $T\asymp \ell $.  
\end{proof}

Taking $\rho =\max ((h^{1-\delta'}\ell^{-1})^{\frac{1}{2}},\, \ell)$ to satisfy (\ref{eqn-3.16}) and (\ref{eqn-3.18}), and thus $\sigma=\rho$ we get
\begin{corollary}\label{cor-3.9}
The following estimates hold as $|x_1|+|y_1|\le C\sigma \ell$
\begin{gather}
e^{1,\T}_{h} (x,y,\tau)=O(h^{1-d}) \ \ \text{for\ \ }   \left\{\begin{aligned}
&\ell \ge h^{\frac{1}{2}+\delta} &&d\ge 4,\\
&\ell \ge h^{\frac{1}{2}-\delta} &&d=3,\\
&\ell \ge h^{\frac{1}{3}-\delta} &&d=2.
\end{aligned}\right.
\label{eqn-3.26}
\end{gather}
\end{corollary}

\begin{remark}\label{remark-3.10}
\begin{enumerate}[wide,label=(\roman*),  labelindent=0pt]
\item
Then $\sigma \ge h^{\frac{1}{3}-\delta'}$.
\item
The similar arguments work under assumption $|x'-y'|\ge \epsilon |x_1|+|y_1|$ instead of $|x_1|+|y_1|\le C\sigma \ell$.
\item
The similar but simpler arguments work for $e^{0,\T}_h(x,y,\tau)$  as $\rho^2 \ell ^0 \ge h^{1-\delta}$.
\end{enumerate}
\end{remark}

\chapter{Successive approximations}
\label{sect-4}

\section{Successive approximations inside domain}
\label{sect-4.1}

\subsection{Standard successive approximations}
\label{sect-4.1.1}
In this Section we are going to apply method of successive approximations to derive asymptotics of $e_h(x,y,\tau)$ near boundary. However we start from the partial proof of Theorem~\ref{thm-1.1} by this method away from the boundary. According to \cite{monsterbook}, Section~\ref{monsterbook-sect-5-3} (which is our standard reference here), we consider problem for propagator $u _h(x,y,t)$
\begin{align*}
&(hD_t-A) u_h =0, && u|_{t=0}=\updelta (x-y)
\end{align*}
with $A=A(x,hD_x,h)$ and rewrite  for $u^\pm_h(x,y,t)\coloneqq u_h(x,y,t)\uptheta(\pm t)$. We have equation
\begin{gather}
(hD_t-A)u_h^\pm = \mp ih\updelta (x-y)\updelta (t)
\label{eqn-4.1}
\end{gather}
which we are going to solve by the successive approximations with unperturbed operator $\bar{A}=a(y,hD_x)$. 
Then (\ref{eqn-4.1}) yields the equality
\begin{gather}
(hD_t -\bar{A})u_h^\pm = \mp ih\updelta (x-y)\updelta (t)+Ru_h^\pm
\label{eqn-4.2}\\
\shortintertext{and hence}
u_h^\pm =\mp ih\bar{G}^\pm \updelta (x-y)\updelta (t)+\bar{G}^\pm Ru_h^\pm,
\label{eqn-4.3}
\end{gather}
where $\bar{G}^\pm$ and $G^\pm$ are parametrices of the problems
\begin{align}
&(hD_t-\bar{A})v=f, \qquad\supp(v)\subset\{\pm (t-t_0)\ge 0 \}
\label{eqn-4.4}\\
\shortintertext{and}
&(hD_t-{A})v=f, \qquad\supp(v)\subset \{\pm(t-t_0) \ge 0 \}
\label{eqn-4.5}
\end{align}
respectively with $\supp(f)\subset \{\pm (t-t_0) \ge 0\}$ for some $t_0 \in \bR$; $R=A-\bar{A}$.

Moreover, equation (\ref{eqn-4.1}) yields that
\begin{equation}
u_h^\pm=\mp ihG^\pm \updelta (x-y)\updelta (t).
\label{eqn-4.6}
\end{equation}
Iterating (\ref{eqn-4.3}) $N$ times and then substituting (\ref{eqn-4.6}) we
obtain the equality
\begin{multline}
u_h^\pm = \mp ih \sum _{0 \le n \le N-1 }
(\bar{G}^\pm R)^n\bar{G}^\pm \updelta (x-y)\updelta (t) \mp\\
ih(\bar{G}^\pm R)^nG^\pm \updelta (x-y)\updelta (t).
\label{eqn-4.7}
\end{multline}
Finally, we apply $^t\!Q_y=\,^t\!Q(y, hD_y)$ to the right of (\ref{eqn-4.7}), where $Q$ is an operator with compactly supported symbol, equal $1$ as $a(y,\xi)\le \tau+\epsilon$.
After this cut-off, according to \cite{monsterbook}, Section~\ref{monsterbook-sect-5-3}, norms of the terms in (\ref{eqn-4.7}) in the strip $\{(x,t)\colon |t|\le T\}$ do not exceed $Ch^{-M-n} T^{2n}$ and 
as $T\le h^{\frac{1}{2}+\delta}$ we can ignore the remainder term. Since we need to consider   $T\le c\ell(x,y)$ in the end we arrive to
\begin{gather}
\max(c\ell,\, h^{1-\delta}) \le T \le h^{\frac{1}{2}+\delta}.
\label{eqn-4.8}
\end{gather}
Then as it was shown (see \cite{monsterbook}, simplified (\ref{monsterbook-4-3-33}))
\begin{multline}
F_{t\to h^{-1}\tau} \Bigl(\bar{\chi}_T u_h^\pm(x,y,t)\,^t\!Q_y\Bigr) \\
\equiv 
\mp i\sum_{0\le n\le N}  (2\pi)^{-d-1}h^{-d+n}\int e^{ih^{-1}\langle x-y,\xi\rangle} F_n(y,\xi,\tau ) [q_2(y,\xi,h)]\,d\xi,\quad \pm \Im \tau <0,
\label{eqn-4.9}
\end{multline}
where $F_n =\sum_{\beta\colon |\beta|\le n } F_{n,\beta} \partial_\xi^\beta $ are differential operators  applied to $q_2$, with coefficients $F_{n\beta}$ holomorphic with respect to $\xi,\tau$ as $\Im \tau \ne 0$, and have poles as $\tau \in \bR$; denominators are $(\tau -a(y,\xi))^{2n+1+|\beta|}$. Then
\begin{multline}
F_{t\to h^{-1}\tau} \Bigl(\bar{\chi}_T u_h(x,y,t)\Bigr) \\
\equiv 
\sum_{0\le n\le N}  (2\pi)^{-d-1}h^{-d+n}\int e^{ih^{-1}\langle x-y,\xi\rangle} \cF_n(y,\xi,\tau ) \,d\xi,\qquad \tau\in \bR
\label{eqn-4.10}
\end{multline}
where $ \cF_n(y,\xi,\tau )=i\bigl( F_{n0} (y,\xi,\tau -i0)- F_{n0} (y,\xi,\tau +i0)\bigr)$. Then the principal term in $e^\T_{T,h}(x,y,\tau)$ can be rewritten as $e^\W_h(x,y,\tau)$ while all other terms as 
\begin{gather}
h^{-d+n-|\alpha|  }  \sum_{\alpha\colon |\alpha|\le 2n-1} (x-y)^\alpha  \int _{\Sigma (y,\tau)} W_{n,\alpha} (y,\xi) e^{ih^{-1}\langle x-y,\xi\rangle}\,d\xi
\label{eqn-4.11}
\end{gather}
with  $n\ge 1$ and smooth $W_{n,\alpha} $. 

Indeed, taking $\xi$-partition and using $\xi$-micro\-hyperbolicity, we can rewrite on each element $\cF_n$ as a sum of  ${W'_{n,\alpha} (y,\xi)\partial_\xi^\alpha \updelta '(\tau-a(y,\xi))}$. Then, integrating by parts we will get (\ref{eqn-4.11}) for terms in $e^\T_{T,h} (x,y,\tau)$.

Now due to strong convexity and stationary phase principle these terms do not exceed
\begin{gather}
Ch^{-d+n-|\alpha|}  ({h}/{\ell})^{\frac{d-1}{2}} \ell^{|\alpha|} \le Ch^{-\frac{d+1}{2}+n -|\alpha| } \ell^{-\frac{d-1}{2} +|\alpha|},
\label{eqn-4.12}
\end{gather}
which is $O(h^{1-d})$ as $d\ge 3$ and $O(h^{-\frac{3}{2}}\ell^{\frac{1}{2}})$ as $d=2$. Recall that $\ell\le h^{\frac{1}{2}+\delta}$.

\subsection{Improved successive approximations}
\label{sect-4.1.2}
To improve this approach we need to freeze symbol at point $w=\frac{1}{2}(x+y)$ and to do this let us observe that \begin{align*}
&\bigl(hD_t-A(x,hD_x,h)\bigr) u_h(x,y,t)=0,\\
& \bigl(hD_t-  A(y,-hD_y,h)\bigr)u_h(x,y,t)=0, && u_h|_{t=0}=\updelta (x-y)
\end{align*}
and therefore denoting $v_h (w,z,t) = u_h (w+\frac{1}{2}z, w-\frac{1}{2}z, t)$ we have
\begin{gather}
\bigl(hD_t - \mathfrak{A} (w, z, hD_z, hD_w,h)\bigr)v_h(w,z,t)=0, \qquad v_h|_{t=0} =  \updelta (z)
\label{eqn-4.13}\\
\shortintertext{with}
\mathfrak{A}  = \frac{1}{2}\bigl(A(w +\frac{1}{2}z, hD_z+\frac{1}{2}hD_w,h) + A(w-\frac{1}{2}z, hD_z-\frac{1}{2}hD_w, h)\bigr).
\label{eqn-4.14}
\end{gather}
Taking as unperturbed operator
\begin{gather}
\bar{\mathfrak{A}}= A(w, hD_z,h) 
\label{eqn-4.15}
\end{gather}
we get $R=\mathfrak{A}-\bar{\mathfrak{A}} $ a a sum of $z^\alpha (h\partial_w)^\beta$ with $|\alpha|+|\beta|\ge 2$. 

With this modification we can write formula, similar to (\ref{eqn-4.7}).  Due to the propagation results we can add to $R$ factor $\bar{\chi}_{cT}(h\partial _w)$ and then each new term in this formula adds an extra factor $C(T^2+h) Th^{-1}$ instead of $CT^2h^{-1}$ and we can replace (\ref{eqn-4.8}) by a weaker condition
\begin{gather}
\max(c\ell,\, h^{1-\delta}) \le T \le h^{\frac{1}{3}+\delta}.
\label{eqn-4.16}
\end{gather}
Calculations, similar to those in the standard method lead us to the main term $e^\W(x,y,\tau)$ and to (\ref{eqn-4.11}) replaced by
\begin{gather}
h^{-d+n-|\alpha|  }  \sum_{\alpha\colon 2|\alpha|\le 3n-1} (x-y)^\alpha  \int _{\Sigma (w,\tau)} W_{n,\alpha} (w,\xi) e^{ih^{-1}\langle z,\xi\rangle}\,d\xi.
\label{eqn-4.17}
\end{gather}
and to the same estimate (\ref{eqn-4.12}) albeit with a restrictions $|\alpha|\le (3n-1)/2$, $\ell \le h^{\frac{1}{3}}$, and we get $O(h^{1-d})$ as $d\ge 1$. 

\section{Successive approximations near boundary}
\label{sect-4.2}

\subsection{Standard successive approximations}
\label{sect-4.2.1}

We know (see \cite{monsterbook}, Section~\ref{monsterbook-sect-7-2}) that $u_h^0(x,y,t)$ and then $u_h^1(x,y,t)$ for $|t|\le T =h^{\frac{1}{2}+\delta}$, $x_1+y_1\le C_0T$  could be constructed by the method of successive approximations with unperturbed operator $\bar{A}=a(0,y'; hD_x)$ (so the principal part of $A$ is frozen at point $(0,y')$):  it leads us to the expression for $e^\T_h(x,y,\tau)$.

In this case unperturbed term would be as for   $\bar{A}$ in the half-space and it will not exceed $Ch^{-d}(h/\ell)^{{d+1}/2}$ and  and one can prove that the perturbed term would acquire factor $T^2/h$ with $T\asymp \ell$ and it does not exceed  (\ref{eqn-4.12}) (with redefined $\ell(x,y)$ now) and it is $O(h^{1-d})$ as $d\ge 3$ and 
$Ch^{-\frac{3}{2}}\ell ^{\frac{1}{2}}$ as $d=2$. 

Since we assumed that the boundary conditions are either Dirichlet or Neumann, we arrive to the expression (\ref{eqn-1.3}) for $e^{1,\T}_h(x,y,\tau)$ as described in Theorem~\ref{thm-1.2}.

\subsection{Improved successive approximations}
\label{sect-4.2.2}

Let us improve this construction in the same manner as we did inside domain. Consider problem for $u^1_h$:
\begin{align}
& (hD_t-A)u_h^1 =0,\qquad  u^1_h |_{t=0}=0,
\label{eqn-4.18}\\[3pt]
& \eth B u_h^1   =-\eth B u_h^0|_{x\in \partial X},
\label{eqn-4.19}
\end{align}
where $u^0_h(x,y,t)$ satisfies the same equation  in the ``whole space'' and $\eth$ is an operator restriction to $\partial X\ni x$.
Then
\begin{align}
& (hD_t-A)u_h^{1,\pm}  =0,
\label{eqn-4.20}\\[3pt]
& \eth B u_h^{1\,\pm}  =-\eth B u_h^{0\,\pm}\label{eqn-4.21}
\end{align}
where again $u^{j\,\pm}=\uptheta(\pm t)u^j_h (x,y,t)$ and 
\begin{align}
&(hD_t -\bar{A})u_h^{1\,\pm} = Ru^{1\,\pm},
\label{eqn-4.22}\\
&\eth B u_h^{1\,\pm}  =- Bu_h^{0\,\pm},
\label{eqn-4.23}
\end{align}
where $u_h^{0,\pm}$ satisfies (\ref{eqn-4.2}); recall that $B=I$ or $B=D_1$ for Dirichlet and Neumann boundary conditions correspondingly\footnote{\label{foot-4} For more general boundary conditions we would need to replace $B$ by $\bar{B}+R_1$.} 
Then like in \cite{monsterbook}, Section~\ref{monsterbook-sect-7-2}
\begin{gather}
u_h^{1\,\pm} =-\bar{G}^{\prime\,\pm} \eth B  u_h^{0\,\pm} + \bar{G}^\pm Ru_h^{1\,\pm},
\label{eqn-4.24}
\end{gather}
where  $\bar{G}^\pm$ are parametrices for the problems
\begin{align}
&(hD_t-\bar{A})v =f, && \eth Bv =0, && \supp(v)\subset \{\pm (t-t_0)\ge 0\}
\label{eqn-4.25}\\
\intertext{and $\bar{G}^{\prime\,\pm}$ are parametrices for the problems}
&(hD_t-\bar{A})v =0, && \eth Bv =f, && \supp(v)\subset \{\pm (t-t_0)\ge 0\}
\label{eqn-4.26}
\end{align}
respectively with $\supp(f)\subset \{\pm (t-t_0) \ge 0\}$ for some $t_0 \in \bR$; $R= A-\bar{A}$. Moreover,
\begin{gather}
u_h^{1\,\pm} =- G ^{\prime\,\pm} \eth B  u_h^{0\,\pm} 
\label{eqn-4.27}
\end{gather}
where ${G}^\pm$ and ${G}^{\prime\,\pm}$ are parametrices for the problems (\ref{eqn-4.25}) and (\ref{eqn-4.26}) but for operator $A$.
Iterating (\ref{eqn-4.24}) $N$ times and then substituting (\ref{eqn-4.27}) we arrive to formula similar to (\ref{eqn-4.7})
\begin{gather}
u_h^{1\,\pm} =-\sum_{0\le n\le N-1}
(\bar{G}^\pm R)^n \bar{G}^{\prime\,\pm} \eth B  u_h^{0\,\pm} - (\bar{G}^\pm R)^N  G ^{\prime\,\pm} \eth B  u_h^{0\,\pm}.
\label{eqn-4.28}
\end{gather}
What is more, we plug $u_h^{0\,\pm}$ given by (\ref{eqn-4.7}) with  $\bar{G}^{\pm}$ and $G^{\pm}$ replaced by $\bar{G}^{0\,\pm}$ and $G^{0\,\pm}$
which are parametrices for (\ref{eqn-4.4}) and (\ref{eqn-4.5}) in the ``whole space''.

Finally, we apply $^t\!Q_1=\,^t\!Q(y_1, z', hD'_z)$ to the right of (\ref{eqn-4.27}) and (\ref{eqn-4.7}), where $Q_1$ is an operator with the  symbol, supported in
\begin{gather}
\Omega_{ T, \rho,\sigma }\coloneqq \{(x_1, z',\zeta)\colon |z'|\le T,\ |\zeta- \bar{\xi}'| \le \rho, \ |x_1|\le \sigma T\}
\label{eqn-4.29}
\end{gather}
where $|a(0,\bar{x}', 0,\bar{\xi}')-\tau)|\le c\rho^2$. Then due to Proposition~\ref{prop-3.3}, under assumptions (\ref{eqn-3.11}) and $\textup{(\ref*{eqn-3.10})}_{1,2}$ we can insert
$Q_2 (x_1, hD'_z)$ to the right of each copy of $R$ with symbol supported in $\Omega_{3T, 3\rho, 3\sigma}$ and equal $1$ in $\Omega_{2T, 2\rho,2\sigma}$.

So far we we followed exactly \cite{monsterbook}, Section~\ref{monsterbook-sect-7-2}, except there we had $\rho=c$. However now instead of $\bar{A}=a(0,y'; hD_x)$ we take
\begin{gather}
\bar{A}=a(w,  hD_x), \qquad w =(0,w'),\qquad w'=\frac{1}{2}(x'+y').
\label{eqn-4.30}
\end{gather}

Then, the norm of $\bar{G}^{\pm}RQ_2$  does not exceed $C(T^2+\sigma T+h) Th^{-1}$ and we can replace (\ref{eqn-4.8}) by a weaker condition 
$(T^2+\sigma T +h) T\le h^{1+\delta}$, which is equivalent to (\ref{eqn-4.16}) plus 
\begin{gather}
\sigma \ell ^2 \le h^{1+\delta}.
\label{eqn-4.31}
\end{gather}

\begin{remark}\label{remark-4.1}
 Recall that $\sigma $ is defined by (\ref{eqn-3.14}) and with $\rho = h^{1-\delta}\ell^{-1}$ condition (\ref{eqn-4.31}) is also equivalent to (\ref{eqn-4.16}).
 However, for us it is more important which value of the reflection angle is allowed in this method.
\end{remark}

Thus we arrive to the following statement:
\begin{proposition} \label{prop-4.2}
Let $Q=Q(x_1,z',hD'_z)$ be an operator with the symbol $q(x_1,z',\zeta')$ supported in $\Omega_{T,\rho,\sigma}$ where condition  \textup{(\ref{eqn-3.11})}  is fulfilled  and let conditions \textup{(\ref{eqn-4.16})} and \textup{(\ref{eqn-4.31})} be also fulfilled.

Then we can skip  remainder terms in both   \textup{(\ref{eqn-4.27})} and modified \textup{(\ref{eqn-4.7})}, which  are $O(h^s)$,  
leaving us only with $\bar{G}^\pm$, $\bar{G}^{\prime\,\pm}$ and $\bar{G}^{0\,\pm}$.
\end{proposition}

Thus (\ref{eqn-4.28}) without the last term becomes an asymptotic series. Now let us calculate $e_h^{1,\T} (x,y,\tau)$ under these assumptions. 

\begin{proposition}\label{prop-4.3}
In the framework of Proposition~\ref{prop-4.2} 
\begin{multline}
F_{t\to h^{-1}\tau} \Bigl(\bar{\chi}_T(t) u_h^{1\,\pm} \,^t\!Q_1 \Bigr)\\
\sim 
\sum_{n\ge 0} h^{-d+n}  \int T\hat{\bar{\chi}} \bigl( \frac{(\tau-\tau')T}{h} \bigr) \,d\tau'  \int d\xi'\int _{\gamma_+} d\xi_1 \int_{\gamma_-} d\eta_1\\
\times
e^{ih^{-1}(x_1\xi_1- y_1\eta_1 +\langle x'-y',\xi'\rangle)}
F_{n}(w, \xi',\xi_1,\eta_1,\tau')[q (w,\xi')] 
\label{eqn-4.32}
\end{multline}
for $ \tau \in \bC_\mp $, where $\gamma_\pm$ are closed contours in $\bC_\pm $ passing once
around the poles of $\bigl(\tau - a(z,\xi _1)\bigr)^{-1}$ lying in $\bC_\pm$ and not passing around the poles lying in $\bC_\mp$; for real $z$ and
$\Im \tau \ne 0$ there is no real pole, and $F_n$ are sums of the terms
\begin{multline}
\xi_1^j \eta_1^k W _{n,l,m,p, j,k}(w,\xi ',\tau )\\ 
\times
\bigl(\tau -a(w,\xi _1,\xi ')\bigr)^{-l}  \bigl(\tau -a(w,\eta_1,\xi ')\bigr)^{-m}  \bigl(\tau -b(w,\xi ')\bigr)^{-p}
\label{eqn-4.33}
\end{multline}
with
\begin{gather}
b(w,\xi')= a(w, 0, \xi'),
\label{eqn-4.34}\\[2pt]
l\ge 1, \; m\ge 1,\;  2(l+m +p )\le 3n +j+k+3, \; j\le l,\ k\le m  
\label{eqn-4.35}
\end{gather}
and uniformly smooth $W _{n,l,m,p, j,k}$.
\end{proposition}

\begin{proof}
Proof follows the proof of  Theorem~\ref{monsterbook-thm-7-2-13} of \cite{monsterbook} with the some modifications, mainly introduced in Subsection~\ref{sect-4.1.2}. We rewrote $A$ as $\mathfrak{A}$ by (\ref{eqn-4.14}) and defined $\bar{\mathfrak{A}}$ by (\ref{eqn-4.15}) with the exception that $w'$ and $z'$ are now $(d-1)$-dimensional variables (see (\ref{eqn-4.30})) and we preserve $x_1$ and $y_1$. 

The main part of the perturbation $R$ is  quadratic  with respect to $z', hD'_w$ and linear with respect to $x_1$. 
 In comparison with (\ref{monsterbook-7-2-60}) of \cite{monsterbook} formula (\ref{eqn-4.33}) is simpler, because we have scalar factors here and products collapse into powers.
On the other hand, instead of claiming that $W$ of that proof are holomorphic satisfying (\ref{monsterbook-7-2-38}) of \cite{monsterbook} here we do it in much explicit way since now we need to follow very carefully what are relations between $j,l,l, m, p$ and $n$.

We have operators $a(w, hD_1, hD'_z, hD'_w)$ with symbols which do not depend on $x$. Let us make  $h$-Fourier transform by $x'$ and $t$.

Consider first $u^0_h(x,y,t)$. 
We claim that
\begin{multline}
F_{t\to h^{-1}\tau} \Bigl(\bar{\chi}_T(t) \eth (hD_{x_1})^r u_h^{0\,\pm} \,^t\!Q_1 \Bigr)\\
\sim 
\sum_{n\ge 0} h^{1-d+n}  \int T\hat{\bar{\chi}} \bigl( \frac{(\tau-\tau')T}{h} \bigr) \,d\tau'  \int d\xi'\int_{\gamma_-} d\eta_1\\
\times
e^{ih^{-1}(- y_1\eta_1 +\langle x'-y',\xi'\rangle)}
F^0_{n}(w, \xi', \eta_1,\tau')[q (w,\xi')] 
\label{eqn-4.36}
\end{multline}
for $ \tau \in \bC_\mp$, $r=1$  and $F^0_n$ are sums of the terms
\begin{gather}
\eta_1^k W^0 _{n,m, j,k}(w,\xi ',\tau ) \bigl(\tau -a(w,\xi _1,\xi ')\bigr)^{-l}  \bigl(\tau -a(w,\eta_1,\xi ')\bigr)^{-m} 
\label{eqn-4.37}
\end{gather}
with $2m \le 3n +k+2-r$ and $ k< m $ with the only exception $n=0$, $m=1$, $k=r=1$. 

\begin{enumerate}[wide,label=(\alph*),  labelindent=0pt]
\item
Like in that proof of \cite{monsterbook},  each factor $z_r$ and $hD_{w_r}$  with $r =2,\ldots, d$ is moved to the right towards $\updelta (z')$,  using commutator relations
\begin{align}
&z_r\bar{G}^{0\,\pm} = \bar{G}^{0\,\pm}  z_r -  \bar{G}^{0\,\pm} [ \bar{\mathfrak{A}},z_r ]\bar{G}^{0\,\pm} 
\label{eqn-4.38}
\end{align}
which also hold for  $hD_{w_r}$. So, each such commuting adds factor $h$ and increases either $m$ or $p$ by $1$. However, each pair of  those factors are accompanied by  $\bar{G}^{0\,\pm}$, and therefore in this process  increment  of the power of $h$ by $1$ is ``paid'' by  increment   of $m$ with the factor $3/2$\,\footnote{\label{foot-5} Compare with the standard method, when the factor was $2$.}. On the other hand, commuting with $^t\!Q_1$, each increment  power of $h$ by $1$ is ``paid'' by increment of $(m+p)$ with the factor $1$.

\item
Also recall that  each factor $x_1$   is moved to the left, towards $\eth B$, also using   (\ref{eqn-4.38}) but for $x_1$.
In this commutator factor $h\eta_1$ is also gained and $m$ is increased by $1$   and also $\bar{G}^{0\,\pm}$ is applied so in the end  $m$ is increased by $2$. When $x_1$ reaches $\eth B$ the term  disappears for Dirichlet boundary condition, or this factor cancels with $B=h\partial_{x_1}$ for Neumann boundary condition (and factor $h$ is gained). process each increment  power of $h\eta_1$ by $1$ is ``paid'' by  increment   of $m$ with the factor $2$. Therefore (\ref{eqn-4.36})--(\ref{eqn-4.37}) has been proven. 

\item
Then for 
\begin{gather}
F_{t\to h^{-1}\tau} \Bigl(\bar{\chi}_T(t) \bar{G}^{\prime\,\pm}\eth   B  u_h^{0\,\pm} \,^t\!Q_1 \Bigr)
\label{eqn-4.39}
\end{gather}
formula (\ref{eqn-4.33}) holds with $l=1$, $p=0$  and $j=1,0$ for Dirichlet and Neumann boundary conditions respectively. 

\item
Again  each factor $z_r$ and $hD_{w_r}$  with $r =2,\ldots, d$ is moved to the right towards $\updelta (z')$,  using commutator relations
(\ref{eqn-4.38}) for $\bar{G}^\pm$ rather than for $\bar{G}^{0\,\pm}$ and also
\begin{align}
&z_r \bar{G}^{\prime\,\pm} = \bar{G}^{\prime\,\pm} z_r -  \bar{G}^\pm [ \bar{\mathfrak{A}},z_r]\bar{G}^{\prime\,\pm}
\label{eqn-4.40}
\end{align}
which  holds also for $hD_{w_r}$. Observe that  that $p$ can increase but  again as in (a) increment  of the power of $h$ by $1$ is ``paid'' by  increment   of $m+p$ with the factor $3/2$\,\footnote{\label{foot-6} Indeed, it is sufficient to consider boundary value problem for ODE $(\tau-D^2)v=f$ with $f(x)= \int_{\gamma_+}\xi ^j (\tau-\xi^2)^{-l}e^{ix \xi }\,d\xi$. One can see easily that
\begin{gather*}
v(x)= \int_{\gamma_+}\Bigl( \xi ^j (\tau-\xi^2)^{-l-1} - \kappa \xi^{j'} \tau ^{(j-j')/2 -l } (\tau- \xi^2)^{-1} \Bigl) e^{ix \xi }\,d\xi
\end{gather*}
satisfies this equation and $v(0)=0$ for some constant $\kappa = \kappa_{\D jl}$ and  $j'\equiv j \mod 2$ and $j'=0,1$. Also $v'(0)=0$ with some other constant
$\kappa = \kappa_{\N jl}$.}.

\item
Also recall that  each factor $x_1$ which is to the left from $\bar{G}^{\prime\,\pm}\eth B$  is moved to the right, towards $\bar{G}^{\prime\,\pm}\eth B$ but instead of 
(\ref{eqn-4.38}) we use
\begin{align}
&x_1\bar{G}_* ^\pm\ = \bar{G}_\D^\pm x_1 -  \bar{G}_\D^\pm [ \bar{\mathfrak{A}},x_1]\bar{G}_*^\pm,
\label{eqn-4.41}\\
&x_1\bar{G}_*^{\prime\,\pm} =   -  \bar{G}_\D^\pm [ \bar{\mathfrak{A}},x_1]\bar{G}_*^{\prime\,\pm}
\label{eqn-4.42}
\end{align}
with $*=\D,\N$ (in particular, for $\bar{G}_\D$ (\ref{eqn-4.38}) holds). 

Again, similarly to  (c) either $l+p$ is increased by $2$ and $j$ does not change, or $j$ is decreased by $1$ and $l+p$ is increased by $1$.

\item
Then we have $n=n'+j+k$ with $l+p=l' +2j$, $m=m'+2k$ and $l'+m' \le \frac{3}{2}n'+2$. 
Those are ``the worst case  scenarios'' when we consider the main part of perturbation $R$, quadratic by $(x'-y')$ and linear by $x_1$. For the rest of perturbation increments of $l+p,m$ are relatively smaller. In the end we arrive to (\ref{eqn-4.35}).

\item
One can see easily that for the main term we have $l=m=1$, $p=0$ and $j+k=1$ with $j=0,1$ for Dirichlet and Neumann problems respectively and $W_{*}=\const $. 
\end{enumerate}
\vspace{-\baselineskip}
\end{proof}

\begin{proposition}\label{prop-4.4}
In the framework of Proposition~\ref{prop-4.2}
\begin{multline}
F_{t\to h^{-1}\tau} \Bigl(\bar{\chi}_T(t) u_h^{1\,\pm} \,^t\!Q_1 \Bigr)\\
\sim 
\sum_{n\ge 0, j\ge 0, k\ge 0} h^{-d+n}  \int T\hat{\bar{\chi}} \bigl( \frac{(\tau-\tau')T}{h} \bigr) \,d\tau'  \int d\xi'\int _{\gamma_+} d\xi_1 \\
\times
e^{ih^{-1}((x_1+y_1)\xi_1+\langle x'-y',\xi'\rangle)}
F_{n,j,k}(w, \xi',\xi_1,\tau')[q (w,\xi')] 
\label{eqn-4.43}
\end{multline}
for $\tau\in \bC_\mp$  where $F_{n,j,k}$ are sums of the terms
\begin{gather}
\xi_1^j \eta_1^k W _{n,l,j,k}(w,\xi ',\tau ) \bigl(\tau -a(y',\xi _1,\xi ')\bigr)^{-l}  
\label{eqn-4.44}
\end{gather}
with $1\le l\le \frac{3}{2}n +j+k+1$ and uniformly smooth $W _{n,l,j,k}$ and also
\begin{multline}
F_{t\to h^{-1}\tau} \Bigl(\bar{\chi}_T(t) u_h^{1} \,^t\!Q_1 \Bigr)\\
\sim 
\sum_{n\ge 0, j\ge 0, k\ge 0} h^{-d+n}  \int T\hat{\bar{\chi}} \bigl( \frac{(\tau-\tau')T}{h} \bigr) \,d\tau'  \int d\xi\\
\times
e^{ih^{-1}((x_1+y_1)\xi_1+\langle x'-y',\xi'\rangle)}
\cF_{n,j,k}(w, \xi',\xi_1,\tau')[q (w,\xi')],
\label{eqn-4.45}
\end{multline}
where $\cF_{n,j,k}(w, \xi',\xi_1,\tau') = F_{n,j,k}(w, \xi',\xi_1,\tau'-i0) - F_{n,j,k}(w, \xi',\xi_1,\tau'+i0)$.
\end{proposition}

\begin{proof}
(\ref{eqn-4.43})--(\ref{eqn-4.44})  follows  from Proposition~\ref{prop-4.3}  by calculating residues at zeroes of $(\tau -a(x,\xi_1,\xi'))$ and  $(\tau -a(x,\eta_1,\xi'))$ in (\ref{eqn-4.33}) and (\ref{eqn-4.43}).

Next we observe that in (\ref{eqn-4.43}) we can replace integral along $\gamma_+$ by integral along $\bR$. Then (\ref{eqn-4.45}) follows trivially.
\end{proof}

\begin{proposition}\label{prop-4.5}
Let $\xi$-microhyperbolicity condition be fulfilled at energy level $\tau$ and let 
\begin{gather}
\eff(x)+\eff(y) \le \sigma \ell
\label{eqn-4.46}\\
\shortintertext{with}
\ell(x,y)\le  h^{\frac{1}{3}+\delta}, \quad \sigma = h^{\frac{1}{2}+\delta }\ell^{-\frac{1}{2}}.
\label{eqn-4.47}
\end{gather}
Let $T=T^* \ge  C_0\ell (x,y)$. Then for any $T\colon T^*\le T\le \epsilon$
\begin{enumerate}[wide,label=(\roman*),  labelindent=0pt] 
\item
As $d\ge 3$ estimate 
\begin{gather}
|e^{1,\T}_{T,h} (x,y,\tau) -[\mp]  e_h^{0,\W}(x,\tilde{y},\tau)|\le Ch^{1-d}
\label{eqn-4.48}
\end{gather}
with sign $[\mp ]$ for Dirichlet and Neumann boundary conditions respectively.
\item
As $d=2$ estimate 
\begin{gather}
|e^{1,\T}_{T,h} (x,y,\tau) -[\mp]  e_h^{0,\W}(x,\tilde{y},\tau)- e_{h,\corr} (x,y,\tau)|\le Ch^{-1}
\label{eqn-4.49}
\end{gather}
with $e_{h,\corr} (x,y,\tau)$ given by \textup{(\ref{eqn-4.54})} below.
\end{enumerate}
\end{proposition}

\begin{proof}
\begin{enumerate}[wide,label=(\alph*),  labelindent=0pt] 
\item
Consider first $\ell \le h^{\frac{1}{2}-\delta}$. Let us take $\rho =\min ( (h/\ell)^{\frac{1}{2}}h^{-\delta} ,\, C_2)$.  Then $\rho\ge C_0\ell$,  $\sigma=\rho^{\frac{1}{2}}$ and
\begin{gather}
\sigma \ell^2 \le h^{1+\delta} \qquad \text{as\ \ } \ell\le h^{\frac{3}{7}+\delta}
\label{eqn-4.50} 
\end{gather}
and the method of successive approximations works. Let us take take $\rho$-admissible partition of unity by $\zeta'$ in $\rho$-vicinity of $\Sigma(w,\tau)$. 
In this case $\rho \ge \ell$ and $\sigma=\rho^{\frac{1}{2}}$. 
Moreover, as $\rho=C_2$ it covers the whole domain
$\{\zeta'\colon b(w,\zeta')<2\tau\}$ and as $\rho < C_2$ and therefore $\ell \ge h^{1-2\delta}$ and also $\rho^2 \ell \ge h^{ 1-2\delta}$ we can use the fact that

\begin{claim}\label{eqn-4.51}
If $\rho \ell \ge h^{1-\delta}$, $\rho \ge C_0\ell^2$, (\ref{eqn-3.11}) is fulfilled  and $Q_1$ is an operator with the symbol equal $0$ in $\rho$-vicinity of $\Sigma(w,\tau)$ and $\ell\ge h^{1-\delta}$ then
\begin{gather*}    
e^{1,\T}_h(x,y,\tau) \,^t\!Q_{1,z} =O(h^s)
\end{gather*}
\end{claim}
\vspace{-\baselineskip}
which follows from Section~\ref{sect-3}. Therefore we can take cut-off operator $Q_1=I$. 

Then the main term of the final expression equals to $e^{0,\W}_h(x,\tilde{y},\tau)$ and is $\asymp h^{-d}(h\ell^{-1})^{(d+1)/2}$ and $n$-th term does not exceed
\begin{gather}
Ch^{-d-\delta'}\bigl(\frac{h}{\ell}\bigr)^{(d+1)/2} \times (\rho^{\frac{1}{2}}\ell^2 h^{-1})^{n-1} 
\label{eqn-4.52}
\end{gather}
due to the same microlocal arguments as above. One can see easily that expression  (\ref{eqn-4.52}) is $O(h^{1-d+\delta''})$ as $d\ge 2$ and $n\ge 3$ and $\ell \le h^{\frac{1}{3}+\delta}$. Furthermore, if we consider terms $O(\ell^2)$ in $R$ then the corresponding second term does not exceed 
\begin{gather}
Ch^{-d-\delta'}\bigl(\frac{h}{\ell}\bigr)^{(d+1)/2} \times  \ell^3 h^{-1})^{n-1} = O(h^{1-d+\delta''}).
\label{eqn-4.53}
\end{gather}
Therefore we need to consider only the second term in the successive approximations. It is equal to 
\begin{multline}
 [\mp] e_{h,\corr} (x,y,\tau)\\
\coloneqq  -\frac{1}{2}(2\pi h)^{-d} \int _{\Sigma (w,\tau)} \lambda (w',\xi') (x_1+y_1) e^{ih^{-1}\langle x-\tilde{y},\xi \rangle}\,d\xi :d_\xi a(w,\xi)
\label{eqn-4.54} 
\end{multline}
with
\begin{gather}
\lambda(w',\xi')= a_{x_1} (x_1,w',\xi)|_{x_1=\xi_1=0}.
\label{eqn-4.55}
\end{gather}
Indeed, one can see easily that the second term in the successive approximation for 
$[\mp] F_{t\to h^{-1}\tau} \bigl(\bar{\chi}_T(t)u^{1\,\pm}_h (x,y,t) \bigr)$ is equal to
\begin{multline*}
\pm 2 h^3 (2\pi h)^{-1-d}  \int _{\gamma_+\times \bR} \lambda (w',\xi ') \xi_1 (\tau -a(x,\xi))^{-3} e^{i\langle x-\tilde{y},\xi \rangle} \, d\xi \\
= \mp \frac{i}{2}h^2  (2\pi h)^{-1-d}  \int _{\gamma_+\times \bR} \lambda (w',\xi ') (x_1+y_1)  (\tau -a(x,\xi))^{-2} e^{i\langle x-\tilde{y},\xi \rangle}  \qquad \mp  \Im \tau <0;
\end{multline*}
then the second term in decomposition for $[\mp] F_{t\to h^{-1}\tau} \bigl(\bar{\chi}_T(t)u^{1}_h (x,y,t) \bigr)$
is equal to
\begin{gather*}
-\frac{1}{2}h(2\pi h)^{-d} \int \lambda (w',\xi') (x_1+y_1) e^{ih^{-1}\langle x-\tilde{y},\xi \rangle}\updelta '(\tau-a(w,\xi))\,d\xi 
\end{gather*}
which implies (\ref{eqn-4.55}).

\item
Due  to stationary phase integral  in \ref{eqn-4.54}) does not exceed $Ch^{-d} \bigl(\frac{h}{\ell}\bigr)^{(d-1)/2}$, 
and  expression (\ref{eqn-4.54}) is $O(h^{1-d})$ as $d\ge 3$, $x_1+y_1\le \sigma \ell$.

Due to Corollary~\ref{cor-3.9} as $d\ge 3$ we need to consider only this case $\ell\le h^{\frac{1}{2}-\delta}$. Therefore Statement~(i)  has been proven.

\item
Let  $d=2$. Then  we need to consider also $h^{\frac{1}{2}-\delta}\le \ell \le h^{\frac{1}{3}+\delta}$ and with selected $\rho = (h/\ell)^{\frac{1}{2}}h^{-\delta}$ 
condition $\sigma \ell^2 \le h^{1+\delta}$ may fail. We need to consider lesser $\rho$. Let us pick up $\rho = h^{1-\delta}\ell^{-1}$. 
Then $\rho\ge C_0\ell^2$ and $\sigma \ell^2 \le h^{1+\delta}$ as $\ell\le h^{\frac{1}{3}-\delta}$.

However condition  $\rho \ge C_0\ell$ may fail and then we need to take care of (\ref{eqn-3.11})  at each point   of $ \Sigma (w,\tau)\cap \{\xi_1=0\}$. 
 
 The good news is that this set  consists of two points $\bar{\xi}^{\prime\pm}$ and we do not need $\rho^2 \ell \ge h^{1-\delta}$ because we can deal with each of these two points separately. Therefore we can replace  $\rho$-admissible operator $Q_1$ with the symbol, equal $1$ in $\rho$-vicinity of $\bar{\xi}^{\prime\pm}$ by operator with the symbol equal $1$ in $\epsilon$-vicinity of this point. Observing that the required gauge transformations are 
 $\xi_2\mapsto \xi_2 - h^{-1}\alpha_\pm z_2$, that are multiplications by $e^{ih^{-1}\alpha_\pm z_2^2/2}$ which do not affect $e(x,y,\tau)$ as $z_2=\frac{1}{2}(x_2-y_2)$. Therefore again we can take $Q_2=I$.

 Then  $n$-th term does not exceed (\ref{eqn-4.52}) which is $O(h^{-1+\delta''})$ as $n\ge 3$ in current settings. Again, if we consider $O(\ell^2)$  terms  in $R$ then the corresponding second term does not exceed (\ref{eqn-4.53}). Thus again we are left with expression (\ref{eqn-4.54}) but this time we cannot claim that it is $O(h^{-1})$ even for
 $\ell \le h^{\frac{1}{2}+\delta}$ and $\rho = h^{1-\delta'}\ell^{-1}$. Therefore Statement~(ii)  has been proven.
\end{enumerate}
\end{proof}

\chapter{Geometric optics method}
\label{sect-5}

\section{Constructing solution}
\label{sect-5.1}

\begin{proposition}\label{prop-5.1}
Let
\begin{gather}
\eff(x)+\eff(y)\ge C_0 \ell^2(x,y)
\label{eqn-5.1}
\end{gather}
and $\ell(x,y)\le \epsilon$. Then on the energy level $\tau$ there exists a single generalized Hamiltonian billiard  of the length $\le c\epsilon$ with at least one reflection from $\partial X$, from $y$ to  $x$ and one from $x$ to $y$, and each  has exactly one reflection and the incidence angles are $\asymp \bigl(\nu(x)+\nu(y)\bigr)\ell^{-1}(x,y)$ (so, they are   \emph{standard  Hamiltonian billiards}).
\end{proposition}

\begin{proof}
The proof is obvious. 
\end{proof}

\enlargethispage{\baselineskip}

\begin{remark}\label{remark-5.2}
Assumption (\ref{eqn-5.1})  is essential. Indeed, if $\partial X$ is strongly concave\footnote{\label{foot-7} With respect to Hamiltonian trajectories.} then for $\eff(x)+\eff(y)\le C_0 \ell^2(x,y)$ there could be multiple reflections and if $\partial X$ is strongly convex\footref{foot-7} then the incidence angle may be $0$. In the former case there could be also multiple  reflections  and multiple billiard rays from $x$ to $y$ on the same energy level $\tau$ while in the latter case there could be no rays at all. 
\end{remark}

\begin{proposition}\label{prop-5.3}
\begin{enumerate}[wide,label=(\roman*),  labelindent=0pt]
\item
Let $A$ be $\xi$-microhyperbolic on the energy level ${\tau=a(y,\theta)}$ and 
\begin{gather*}
a(x,0)<\tau -\epsilon.
\end{gather*}
Let $\varphi ^0(x,y,\theta)$ be a solution of the stationary eikonal equation   \textup{(\ref{eqn-2.7})} satisfying
\textup{(\ref{eqn-2.8})} and \textup{(\ref{eqn-2.9})}. Then
\begin{multline}
|\partial_{x} ^\alpha \partial_y^\beta \partial_\theta^\gamma \varphi^0(x,y,\theta) |\le C_{\alpha\beta\gamma} 
\left\{\begin{aligned}
&T &&\alpha=\beta =0,\\
&1 && |\alpha|+|\beta|\ge 1
\end{aligned}\right.\\
\forall x,y \colon |x-y| \le T
\label{eqn-5.2}
\end{multline}
provided $T\le \epsilon$ with sufficiently small constant $\epsilon>0$.

\item
Let $\varphi  (x,y) $ be another solution of the same equation satisfying
\begin{gather}
\varphi  |_{x_1=0}= \varphi^0 |_{x_1=0},\qquad \partial_{x_1}\varphi  |_{x_1=0}= - \partial_{x_1}\varphi^0 |_{x_1=0}.
\label{eqn-5.3}\\
\shortintertext{Assume that }
\sigma = |  \partial_{x_1}\varphi  |_{x_1=0}|\ge C_0T.
\label{eqn-5.4}
\end{gather}
Then as $|x-y|\le T$ the following inequalities hold:
\begin{gather}
|\partial_x ^{\alpha}\partial_y^\beta \partial_\theta^\gamma \varphi  |\le C_{\alpha\beta\gamma}  
\left\{\begin{aligned}
&T &&\alpha=\beta =0,\\
&1 && \alpha_1\le 1,\\
&\sigma ^{3-\alpha_1-|\alpha|-|\beta|-|\gamma|}\qquad &&\alpha_1\ge 2.
\end{aligned}\right.
\label{eqn-5.5}
\end{gather}
\item
Furthermore, 
\begin{gather}
|\partial_{x}^{\alpha}\partial_y^\beta \partial_\theta^\gamma \varphi  |\le C_{\alpha\beta\gamma}  
\sigma ^{4-\alpha_1-|\alpha|-|\beta|-|\gamma|}\qquad \alpha_1\ge 2,\ |\gamma'|\ge 1.
\label{eqn-5.6}
\end{gather}
\end{enumerate}
\end{proposition}

\begin{proof}
\begin{enumerate}[wide,label=(\alph*),  labelindent=0pt]
\item
Consider first $\varphi=\varphi^0$. Then the left-hand expression of (\ref{eqn-5.2}) does not exceed $C_{\alpha\beta\gamma}$ and we need to consider only
$\alpha=\beta=0$. Applying $\partial _\theta^\gamma$ to  eikonal equation we see that  $\frac{d}{dt}\partial^\gamma_\theta \varphi$ is bounded 
where here and below
\begin{gather}
\dfrac{d\ }{dt} \coloneqq \bigl(\partial _t -\sum_k a^{(k)}(x, \nabla_x\varphi )\bigr).
\label{eqn-5.7}
\end{gather}
Since $|\partial^\gamma_\theta \varphi^0(x,y,t,\theta)|\le C_\gamma T$ as $t=0$ and $|x-y|\le T$ we conclude that the same is true for $|t|\le T$.

\item
First, we provide this estimate as $x_1=0$; from  (\ref{eqn-5.2}) and (\ref{eqn-5.3})  we conclude that (\ref{eqn-5.5})  holds as $\alpha_1=0,1$.
Consider
\begin{multline}
\varphi_{x_1}^2 = a(y,\theta)+ b(x, \nabla_{x'}\varphi),\\
b(x, \nabla_{x'}\varphi) \coloneqq  -V -\sum_{j\ge 2} g^{jk}(\varphi _{x_j} -V_j)(\varphi _{x_k} -V_k).
\label{eqn-5.8}
\end{multline}
Differentiating once by $x_1$ we get
\begin{gather}  
2\varphi_{x_1}\varphi_{x_1x_1} = b_{x_1} (x, \nabla_{x'}\varphi) + \sum_{k\ge 2} b^{(k)}(x,\nabla_{x'} )\varphi_{x_kx_1}
\label{eqn-5.9}
\end{gather}
which in virtue of (\ref{eqn-5.5}) with $\alpha_1\le 1$ does not exceed $C$ and therefore 
\begin{gather}
|\varphi_{x_1x_1}| \le C\sigma^{-1}.
\label{eqn-5.10}
\end{gather}
Further, applying to (\ref{eqn-5.8}) $\partial_{x}^{\alpha-\upvarepsilon}\partial_y^\beta \partial_\theta^\gamma$ with $\upvarepsilon =(1,0,\ldots,0)$ and $\alpha_1=2$ and plugging $x_1=0$ we see that
$2\varphi_{x_1}\varphi_{\alpha\beta\gamma}$\,\footnote{\label{foot-8} Here and below we use notation  $\varphi_{\alpha\beta\gamma}= \partial_{x,t}^{\alpha}\partial_y^\beta \partial_\theta^\gamma \varphi$.} is a linear combination with bounded coefficients of 
$\varphi_{\alpha^1\beta^1\gamma^1}$ with $|\alpha^1|+|\beta^1|+|\gamma^1|<|\alpha|+|\beta|+|\gamma|$ plus bounded terms, and then by induction by 
$|\alpha|+|\beta|+|\gamma|$ we arrive to (\ref{eqn-5.5}) for $\alpha_1=2$. 

Consider induction by $\alpha_1$. Assume that (\ref{eqn-5.5}) is proven for lesser $\alpha_1$  and for  arbitrary $\alpha ',\beta,\gamma $.

Applying to (\ref{eqn-5.8}) $\partial_{x,t}^{\alpha-\upvarepsilon}\partial_y^\beta \partial_\theta^\gamma$ and plugging $x_1=0$ we get 
linear combinations with bounded coefficients of products 
\begin{gather}
\prod _{1\le j\le m} \varphi_{\alpha^j\beta^j\gamma^j}
\label{eqn-5.11}
\end{gather}
where $\sum_j |\alpha^j|\le |\alpha| +m-1$, $\sum_j \beta^j\le \beta$, $\sum_j\gamma^j\le \gamma$. These terms appear when
\begin{enumerate}[ label=(\Alph*),  labelindent=0pt]
\item
We differentiate $b(x,\nabla_{x'}\varphi)$, in this case $\sum_j \alpha^j_1<\alpha_1$, $\sum_j |\alpha^j|\le |\alpha|$ and these terms due to induction assumption are 
$O(\sigma^{4-\alpha_1 -|\alpha|-|\beta|-|\gamma})$. 

\item
We differentiate $\varphi_{x_1}^2$ and  $m=2$, $|\alpha^1|+|\alpha^2|=|\alpha|+1$, $\alpha^1_1+\alpha_1^2=\alpha_1+1$, $\alpha_1^j<\alpha$ and these terms due to induction also are  $O(\sigma^{4-\alpha_1 -|\alpha|-|\beta|-|\gamma|})$.
\item
We differentiate $\varphi_{x_1}^2$ and  $m=1$, $\alpha^1_1 = \alpha_1$, $\alpha_1^1<\alpha_1$. Then we apply induction by $|\alpha|+|\beta|+|\gamma|$, and these terms due to induction assumption are $O(\sigma^{4-\alpha_1 -|\alpha|-|\beta|-|\gamma|})$ as well. 
\end{enumerate}

Then dividing by $\varphi_{x_1}$ and we get  $O(\sigma^{4-\alpha_1 -|\alpha|-|\beta|-|\gamma|})$. We need base of the induction in (C). However as $|\alpha|=\alpha_1$, 
$\beta=\gamma=0$ terms of type (C) do not appear.

\item
We need to expand these estimates to $x_1>0$. Applying to equation (\ref{eqn-5.8}) $\partial_{x,t}^{\alpha}\partial_y^\beta \partial_\theta^\gamma$ with $|\alpha|+|\beta|+|\gamma|=1$
we see that $\dfrac{d\ }{dt}\varphi_{\alpha\beta\gamma}$ is bounded and therefore in this case we extend   (\ref{eqn-5.5}) from $x_1=0$ to $x_1>0$.

Next, for $|\alpha|+|\beta|+|\gamma|=2$ we get that $\dfrac{d\ }{dt}\varphi_{\alpha\beta\gamma}$ does not exceed $C(S+1)$ with
\begin{gather}
S=\sum_{\alpha,\beta,\gamma\coloneqq |\alpha|+|\beta|+|\gamma|=2} |\varphi_{\alpha\beta\gamma}|^2
\label{eqn-5.12}
\end{gather}
and therefore $|\dfrac{dS^{\frac{1}{2}}}{dt}|\le C(S+1)$; since $S|_{x_1=0}=O(\sigma^{-1})$ due to (\ref{eqn-5.10}) we conclude that for $T\le \epsilon \sigma$ estimate $S^{\frac{1}{2}}\le C\sigma ^{-1}$ holds. Then
\begin{gather}
|\varphi_{\alpha\beta\gamma}|\le C\sigma^{-1}\qquad \text{for\ \ } |\alpha|+|\beta|+|\gamma|=2.
\label{eqn-5.13}
\end{gather}
But then, taking $\alpha_1\le 1$ we get that $\dfrac{d\ }{dt}\varphi_{\alpha\beta\gamma}$ does not exceed $C\sigma^{-1}(S^{\frac{1}{2}}+1)$ with $S$ defined by
(\ref{eqn-5.12}) with summation restricted to $\alpha_1\le 1$ and then  due to (\ref{eqn-5.5}) for $x_1=0$ we conclude that
\begin{gather}
|\varphi_{\alpha\beta\gamma}|\le C\qquad \text{for\ \ } |\alpha|+|\beta|+|\gamma|=2,\ \alpha_1\le 1.
\label{eqn-5.14}
\end{gather}
And then, taking $\alpha_1=0$ we get that $\dfrac{d\ }{dt}\varphi_{\alpha\beta\gamma}$ does not exceed $C$  and then  due to (\ref{eqn-5.5}) for $x_1=0$  we conclude that
\begin{gather}
|\varphi_{\alpha\beta\gamma}|\le CT\qquad \text{for\ \ }\alpha=\beta=0, \ |\gamma|=2.
\label{eqn-5.15}
\end{gather}
So, (\ref{eqn-5.5}) holds as $|\alpha|+|\beta|+|\gamma|\le 2$, $x_1>0$.

\item
Assume that for $|\alpha|+|\beta|+|\gamma|< p$ estimate (\ref{eqn-5.5}) has been proven. Applying to equation (\ref{eqn-5.8}) $\partial_{x,t}^{\alpha}\partial_y^\beta \partial_\theta^\gamma$ with $|\alpha|+|\beta|+|\gamma|=p\ge 3$ we again get a linear combination with bounded coefficients of (\ref{eqn-5.11}) products (plus bounded terms) where this time $\sum_j |\alpha^j|\le |\alpha|+m$. Then, as $\alpha_1=q$, $|\alpha|+|\beta|+|\gamma|=p$ we see from the same analysis as in (ii) that
\begin{gather*}
|\frac{d\ }{dt}\varphi_{\alpha\beta\gamma}|\le C\sigma^{-1}S ^{\frac{1}{2}} + C\sigma^{3-p-q}+C
\end{gather*}
and then (\ref{eqn-5.5}) is extended from $x_1=0$ to $x_1>0$, but for $\alpha=\beta=0$, when we proved so far that $\partial_\theta^\gamma=O(1)$. But then
$|\frac{d\ }{dt}\varphi_{\gamma}|\le C$ and (\ref{eqn-5.5}) is again extended from $x_1=0$ to $x_1>0$.

\item
To prove Statement (iii) we again start from estimates at $x_1=0$. Then, exactly like in (b) we prove see that
\begin{gather*}
\varphi_{x_1}\varphi_{x_1x_1\theta_j} + \varphi_{x_1\theta_j}\varphi_{x_1x_1} =O(1)
\end{gather*} 
and since $ \varphi_{x_1\theta_j}=- \varphi^0_{x_1\theta_j}=O(\sigma)$ as $j\ge 2$, and we know that  $\varphi_{x_1x_1}=O(\sigma ^{-1})$, we conclude that
$\varphi_{x_1x_1\theta_j} =O(\sigma^{-1})$. Again, like in (b), using double induction by $|\alpha|+|\beta|+|\gamma|$ and $\alpha_1$ we prove (\ref{eqn-5.6}) as $x_1=0$.

Finally, like in (c) and (d) we expand this estimate to $x_1>0$. We leave easy details to the reader. 
\end{enumerate}
\vspace{-\baselineskip}
\end{proof}

\begin{proposition}\label{prop-5.4}
Let $\varphi ^0(x,y,\theta)$ and $\varphi  (x,y, \theta)$ be defined in Proposition~\ref{prop-5.3}(i) and (ii) respectively.
Consider  asymptotic solution
\begin{gather}
u^0_h(x,y,t)\sim (2\pi h)^{-d}\int e^{i\Phi ^0(x,y,t,\theta)} \sum_{n\ge 0}  B^0_n(x,y,t,\theta)h^n\,d\theta
\label{eqn-5.16}\\
\intertext{where
$\Phi^0 (x,y,t,\theta)= \varphi^0 (x,y,\theta)+ta(y,\theta)$ and}
\Phi  (x,y,t,\theta)= \varphi  (x,y,\theta)+ta(y,\theta),
\label{eqn-5.17}
\end{gather}
and $B^0_n$ are uniformly smooth and satisfying corresponding transport equations with initial conditions such that 
$u^0_h(x,y,0) = \updelta (x-y)$: 
\begin{gather}
|\partial_{x,t} ^\alpha \partial_y^\beta \partial_\theta^\gamma B^0_n (x,y,t, \theta) |\le C_{\alpha\beta\gamma} \quad
\forall x,y, t\colon |x-y|+|t|\le T\  \forall \alpha,\beta,\gamma.
\label{eqn-5.18}
\end{gather}
Consider  \underline{formal} asymptotic solution
\begin{gather}
U^1_h(x,y,t)=(2\pi h)^{-d}\int e^{i\varphi (x,y,t,\theta)} \sum_{n\ge 0}  B_n(x,y,t,\theta)h^n \,d\theta
\label{eqn-5.19}
\end{gather}
where $B_n$ satisfy corresponding transport equations and \underline{one} of the boundary conditions
\begin{align}
&B_n = -B^0_n    &&\text{as\ \ } x_1=0,
\label{eqn-5.20}\\[3pt]
&\varphi_{x_1} B_n -i\partial_{x_1} B_{n-1} =  \varphi_{x_1} B^0_n +i\partial_{x_1} B^0_{n-1}=0
&&\text{as\ \ } x_1=0,
\label{eqn-5.21}
\end{align}
corresponding to $u^1_h=-u^0_h$ as $x_1=0$ and $\partial_{x_1}u^1_h=-\partial_{x_1} u^0_h$  as $x_1=0$ respectively. 
Then the following inequalities hold:
\begin{gather}
|\partial_{x,t}^{\alpha}\partial_y^\beta \partial_\theta^\gamma B_n |\le C_{n\alpha\beta\gamma}  
\left\{\begin{aligned}
&1 + T\sigma^{-1-|\alpha|- |\beta|-|\gamma|} &&  \alpha_1=n=0,\\
&\sigma^{-3n-\alpha_1-|\alpha|-|\beta|-|\gamma|}\qquad &&\alpha_1+n\ge 1,
\end{aligned}\right.
\label{eqn-5.22}\\
\intertext{and if $|\gamma|\ge 1$} 
|\partial_{x,t}^{\alpha}\partial_y^\beta \partial_\theta^\gamma B_n |\le C_{n\alpha\beta\gamma}  
\left\{\begin{aligned}
&1 + T\sigma^{-|\alpha|- |\beta|-|\gamma|} &&  \alpha_1=n=0,\\
&\sigma^{1-3n-\alpha_1-|\alpha|-|\beta|-|\gamma|}\qquad &&\alpha_1+n\ge 1.
\end{aligned}\right.
\label{eqn-5.23}
\end{gather}
\end{proposition}

\begin{proof}
\begin{enumerate}[wide,label=(\alph*),  labelindent=0pt]
\item
Observe that the transport equation is
\begin{gather}
\Bigl(\frac{d\ }{dt} +f\Bigr) B_n =\cL B_{n-1},
\label{eqn-5.24}
\end{gather}
where $f$ is a linear combination with the smooth coefficients of second derivatives of $\Phi$, 
$\cL=\cL(x,D_x)$ is a second order differential operator with the smooth coefficients, and with the coefficient $1$ at $D_{x_1}^2$, and
$B_{-1}=0$. In virtue of (\ref{eqn-5.5})
\begin{gather}
|D^\alpha_{x,t} D^\beta_yD_\theta^\gamma f|\le \sigma^{-1-\alpha_1-|\alpha|-|\beta|-|\gamma|}.
\label{eqn-5.25}
\end{gather}

Consider first $n=0$. Then equation (\ref{eqn-5.24}) has a right-hand expression $0$ and boundary condition (\ref{eqn-5.21}) becomes
$B_0 =\mp B_0^0$ as $x_1=0$. Let us establish first (\ref{eqn-5.22})  as $x_1=0$. As $\alpha_1=0$ these estimates follow from the above boundary condition  and 
(\ref{eqn-5.18}). Consider transport equation (\ref{eqn-5.24}):
\begin{gather}
\varphi_{x_1} B_{0,x_1}+ \sum_{k\ge 2}b ^{(k)}(x,\nabla_{x'} \varphi) B_{0,x_k} -\frac{1}{2} B_{0,t} +\frac{1}{2} f B_0 =0
\label{eqn-5.26}
\end{gather}
and set $x_1=0$. Then all terms in (\ref{eqn-5.26}) are smooth, except the first and the last one, and the latter satisfies
\begin{gather*}
|\partial^{\alpha -\upvarepsilon}_{x,t}\partial^\beta_y\partial^\gamma_\theta (f\, B_0)|\le C\sigma^{-|\alpha| -|\beta|-|\gamma|}\qquad \alpha_1=1
\end{gather*}
due to (\ref{eqn-5.21}) and smoothness of $B_0$ at $x_1=0$.  Sinnce $\partial^{\alpha }_{x,t}\partial^\beta_y\partial^\gamma_\theta \varphi$ is obtained by division by $\varphi_{x_1}$ we arrive to (\ref{eqn-5.22}) as $n=0$, $\alpha_1=0$.

Next we apply a double induction by $\alpha_1$ and $|\alpha|+|\beta|+|\gamma|$ exactly as in the proof of Proposition~\ref{prop-5.3}, Part (b) with the following  modifications: 
\begin{enumerate}[label=(\Alph*)]
\item
We use transport equation (\ref{eqn-5.24}) rather than eikonal equation.
\item
We observe that $\partial^{\alpha -\upvarepsilon}_{x,t}\partial^\beta_y\partial^\gamma_\theta $, applied to the first  term in (\ref{eqn-5.26}) equals to
$\varphi_{x_1} \varphi_{\alpha\beta\gamma}$ plus a linear combination of  
\begin{gather}
\prod_ {1\le j\le m} \varphi_{\alpha^j \beta^j\gamma^j} \partial^{\alpha^0}_{x,t}\partial^{\beta^0}_y\partial^{\gamma^0}_\theta B_0
\label{eqn-5.27}
\end{gather}
with $m=1$ and 
$\sum_j \alpha^j =\alpha+ \upvarepsilon$, $\sum_j  \beta^j= \beta$, $\sum_j \gamma^j =\gamma$\,\footnote{\label{foot-9} With summation over  $0\le j\le m$} and\\ $|\alpha^0|+|\beta^0|+|\gamma^0|<|\alpha|+|\beta|+|\gamma|$.

\item
We  observe that $\partial^{\alpha -\upvarepsilon}_{x,t}\partial^\beta_y\partial^\gamma_\theta $, applied to the second term in (\ref{eqn-5.26}) also is a linear combination  with bounded coefficients of (\ref{eqn-5.27}) with $\sum_j \beta^j=\beta$, $\sum_j \gamma^j=\gamma$, $\sum_j  |\alpha^j| \le  |\alpha|+m$, $\sum_j  \alpha^j_1 <\alpha_1$\,\footref{foot-9} and $|\alpha^j|+|\beta^j|+|\gamma^j| \ge 2$ for $j\ge 1$; it is possible that $m=0$ here.
\end{enumerate}
We leave simple but tedious arguments to the reader.

\item
Then we extend these estimates to $x_1>0$, $T\le \sigma$. We apply arguments of Parts (c), (d) of the proof of Proposition~\ref{prop-5.3}. First, transport equation (\ref{eqn-5.24}) and boundary condition $B_0=B^0_0$ as $x_1=0$ imply  that $|B_0|\le C$ for $T\le C_0^{-1}\sigma$.  

Then applying $\partial_{x,t}^\alpha \partial_y^\beta\partial_\theta^\gamma$ with $|\alpha|+|\beta|+|\gamma|=1$ to transport equation  (\ref{eqn-5.24})  we get 
\begin{gather}
|\frac{d\ }{dt} B_{0,\alpha\beta\gamma}|\le C\sigma^{-1}S^{\frac{1}{2}} +C \sigma^{-3}, \qquad S\coloneqq \Bigl(\sum_{\alpha,\beta,\gamma} |B_{0,\alpha\beta\gamma}|^2\Bigr)^{\frac{1}{2}}
\label{eqn-5.28}
\end{gather}
due to estimates $|\partial_{x,y,\theta} ^2 \varphi|\le C\sigma^{-1}$ and $|\partial_{x,t, y,\theta} f|\le C\sigma^{-3}$ 
($B_{n,\alpha\beta}\coloneqq \partial_{x,t}^\alpha \partial_y^\beta\partial_\theta^\gamma B_n$) and then 
for $T\le C_0^{-1}\sigma$ we get from this estimate and estimate $|B_{0,\alpha\beta\gamma}|\le C\sigma^{-2}$ at $x_1=0$ to estimate $ |B_{0,\alpha\beta\gamma}|\le C\sigma^{-2}$. So far the only restriction is $|\alpha|+|\beta|+|\gamma|=1$--here and in summation in the definition of $S$.

However, if we add an extra restriction $\alpha_1=0$, then due to this estimate (without restriction) and estimates $|\partial_{x,y, \theta} ^2 \varphi|\le C$ and 
$|\partial_{x,y,t,\theta} f|\le C\sigma^{-2}$ if there is only one derivative with respect to $x_1$, we arrive to
\begin{gather*}
|\frac{d\ }{dt} B_{0,\alpha\beta\gamma}|\le C\sigma^{-1}S^{\frac{1}{2}} +C \sigma^{-2} 
\end{gather*}
where $S$ is defined by (\ref{eqn-5.28}) with summation under the same restriction. Then due to this estimate and estimate $|B_{0,\alpha\beta\gamma}|\le C$ as $x_1=0$ and $\alpha_1=0$ we arrive to the estimate $|B_{0,\alpha\beta\gamma}|\le C + CT \sigma^{-2}$.

We apply induction by $q\coloneqq |\alpha|+|\beta|+|\gamma|$. Assuming that for lesser values it is proven, we due to this induction assumption and estimates (\ref{eqn-5.5}) and (\ref{eqn-5.25}) arrive to
\begin{gather}
|\frac{d\ }{dt} B_{0,\alpha\beta\gamma}|\le C\sigma^{-1}S^{\frac{1}{2}} +C \sigma^{p-1-2|\alpha|+|\beta|+|\gamma|}, \quad S\coloneqq \Bigl(\sum_{\alpha,\beta,\gamma} |B_{0,\alpha\beta\gamma}|^2\Bigr)^{\frac{1}{2}}
\label{eqn-5.29}\\
\intertext{under restrictions}
|\alpha|+|\beta|+|\gamma|=q,\quad \alpha_1\le |\alpha|-p
\label{eqn-5.30}\\
\intertext{with $p=0$ and then we recover estimate}
|B_{0,\alpha\beta\gamma}|\le C\sigma^{p-2|\alpha|-|\beta|-|\gamma|}
\label{eqn-5.31}
\end{gather}
under these  restrictions. Based on it we arrive to (\ref{eqn-5.29})  then to (\ref{eqn-5.31}) now under restriction (\ref{eqn-5.30}) with $p=1$; again due to (\ref{eqn-5.31}) at $x_1=0$, and so on for $p=2,\ldots$ until we reach $p=|\alpha|$ but in the latter case we use $|B_{0,\alpha\beta\gamma}|\le C_{\alpha\beta\gamma}$ at $x_1=0$ as $\alpha_1=0$ and thus we achieve some improvement over (\ref{eqn-5.31}): namely we get  (\ref{eqn-5.22}) as $n=\alpha_1=0$. We leave easy but tedious details to the reader.

\item
Next we apply induction by $n\ge 1$ to estimate $B_{n,\alpha\beta\gamma}$ at $x_1=0$.  However,  first we consider $n=1$ to establish the pattern. Observe first that under condition (\ref{eqn-5.20}) $B_n=-B^0_n$ and therefore 
\begin{gather}
|B_{n,\alpha\beta\gamma}|\le C_{\alpha\beta\gamma}\qquad \text{at\ \ } x_1=0 \text{\ \ as\ \ } \alpha_1=0,
\label{eqn-5.32}\\
\shortintertext{ but under condition (\ref{eqn-5.21}) }
B_n=-B_n^0 +i \varphi_{x_1}^{-1}\bigr(B_{n-1,x_1}+B^0_{n-1,x_1}\bigr)\qquad \text{at\ \ } x_1=0
\notag
\end{gather}
and therefore estimate 
(\ref{eqn-5.22}) holds as $n=1$, $\alpha_1=0$, $x_1=0$. It follows from estimates for $|\partial_{x_1} B_{0,\alpha\beta\gamma}|\le C_{\alpha\beta\gamma} \sigma^{-2-|\alpha|-|\beta|-|\gamma|}$ as $\alpha_1=0$, $x_1=0$ (but we divide by $\varphi_{x_1}$ which brings an extra factor $\sigma^{-1}$).

Recall that the transport equation is
\begin{gather}
(\frac{d\ }{dt} +f )B_n= G_n,\qquad G_n\coloneqq \cL B_{n-1},
\label{eqn-5.33}\\
\intertext{and due to results of Part (b)}
|G_{n,(\alpha-\upvarepsilon)\beta\gamma}|\le C\sigma^{1-3n -\alpha_1 -|\alpha|-|\beta|-|\gamma|}
\label{eqn-5.34}
\end{gather}
and therefore in both cases estimate (\ref{eqn-5.22}) holds (as $n=1$, $\alpha_1=1$ and $x_1=0$). 

Applying the same argument as in Part (b) we can extend  (\ref{eqn-5.34}) and   (\ref{eqn-5.22}) to arbitrary  $\alpha_1$ as $n=1$, and using induction by $n$ to arbitrary $n$ as well. Again, we leave easy but tedious details to the reader. 

\item
Using the same arguments as in Part (e) of the proof of Proposition~\ref{prop-5.3}, we can improve these estimates to (\ref{eqn-5.23}) as $|\gamma |\ge 1$. 
\end{enumerate}\vspace{-\baselineskip}
\end{proof}

\begin{remark}\label{remark-5.5}
Under assumption (\ref{eqn-5.20}) we can further improve estimates (\ref{eqn-5.22}) and (\ref{eqn-5.23}) for $n\ge 1$. However it would not improve our final result.
\end{remark}

\begin{remark}\label{remark-5.6}
One can wonder how sharp are our estimates. 
\begin{enumerate}[wide,label=(\roman*),  labelindent=0pt]
\item
Due to 
(\ref{eqn-5.2}) and (\ref{eqn-5.3}) $\varphi_{x_1x_1}\asymp \sigma^{-1}$ as $b_{x_1} (x, \nabla_{x'}\varphi) \asymp 1$. Indeed
\begin{align*}
&2\varphi^0_{x_1}\varphi^0_{x_1x_1} + \sum_{k\ge 2}b^{(k)}(x,\nabla_{x'} \varphi^0) \varphi^0_{x_kx_1}-\varphi^0_{x_1t} +b_{x_1}(x,\nabla_{x'} \varphi^0)=0,\\
&2\varphi_{x_1}\varphi_{x_1x_1} + \sum_{k\ge 2}b^{(k)}(x,\nabla_{x'} \varphi ) \varphi _{x_kx_1}-\varphi _{x_1t} +b_{x_1}(x,\nabla_{x'} \varphi )=0,
\end{align*}
and since $\varphi^0$ is smooth function, we conclude that
\begin{gather*}
\sum_{k\ge 2}b^{(k)}(x,\nabla_{x'} \varphi^0) \varphi^0_{x_kx_1}-\varphi^0_{x_1t} +b_{x_1}(x,\nabla_{x'} \varphi^0)=O(\sigma)\\
\intertext{and then due to (\ref{eqn-5.3}) that}
\varphi_{x_1}\varphi_{x_1x_1} + b_{x_1}(x,\nabla_{x'} \varphi )=O(\sigma)\qquad \text{as\ \ } x_1=0.
\end{gather*}
Therefore estimates (\ref{eqn-5.5}) cannot be improved, at least for derivatives only with respect to $x_1$. 

\item
Repeating construction of Proposition~\ref{prop-5.3}, we conclude that  under the same assumption $b_{x_1} (x, \nabla_{x'}\varphi) \asymp 1$,  
$\partial_{x_1}^k \varphi \asymp \sigma^{3-2k}$ for all $k\ge 2$.
Then from transport equation (\ref{eqn-5.24}) we  we conclude that $B_{0,x_1}\asymp \sigma^{-2}$ and repeating construction of Proposition~\ref{prop-5.4} we conclude that
$\partial_{x_1}^k B_0 \asymp \sigma^{-2k}$ for all $k\ge 1$ and finally 
$\partial_{x_1}^k B_n \asymp \sigma^{-3n-2k}$ for all $n\ge 0$, $k\ge 1$.
\end{enumerate}
\end{remark}

\begin{proposition}\label{prop-5.7}
\begin{enumerate}[wide,label=(\roman*),  labelindent=0pt]
\item
In the framework of Proposition~\ref{prop-5.3}(i) as $|\alpha|\le 1$ 
\begin{align}
&|\partial_{x,t}^\alpha \partial_\theta^\gamma\bigl(\Phi^0-\bar{\Phi}^0 \bigr)|\le C_\gamma T^{2-|\alpha|} &&\text{with\ \ } \bar{\Phi}^0\coloneqq \langle x-y,\theta \rangle + t a(y,\theta).
\label{eqn-5.35}
\end{align}
\item
In the framework of Proposition~\ref{prop-5.3}(ii) as $|\alpha|\le 1$, $\alpha_1=0$
\begin{multline}
|\partial_{x,t}^\alpha \partial_\theta^\gamma\bigl(\Phi\ -\bar{\Phi}\  \bigr)|\le C_\gamma \bigl(\sigma^{1-|\alpha|-|\gamma|}T^2 + T^{2-|\alpha|}\bigr)\\
\text{with\ \ }
\bar{\Phi}\ \coloneqq \langle x-\tilde{y},\theta \rangle + t a(0,y',\theta),
\label{eqn-5.36}
\end{multline}
where $\tilde{y}=(-y_1,y')$ as $y=(y_1,y')$.
\end{enumerate}
\end{proposition}

\begin{proof}
Proof of both statements  follows from decomposition of  $\partial ^\alpha _{x,t}\partial_\theta^\gamma \Phi^0$ into Taylor series with quadratic error.  First we prove (i) using estimates (\ref{eqn-5.2}) and then (ii) using estimates (\ref{eqn-5.5}).

For $d=2$ we will need a better approximation; it will be proven later.
\end{proof}

Then we have immediately
\begin{corollary}\label{corollary-4.8}
\begin{enumerate}[wide,label=(\roman*),  labelindent=0pt]
\item
In the framework of Proposition~\ref{prop-5.3}(i) for $Ch\le \ell^0(x,y)+|t|\le \epsilon$ 
\begin{gather}
a(x,\nabla_x \varphi^0 )=\tau,\quad  \nabla_\theta \Phi^0 (x,y,t,\theta) =0
\label{eqn-5.37}
\end{gather}
has  no more than a single solution $\theta$ and if it has, then
$|t|\asymp \ell^0$ and $\nabla^2_\theta \Phi^0 $ is a positive definite matrix. This solution $\theta=\bar{\theta}+O(T)$ where 
$\bar{\theta}$ is a solution to $t\nabla_\theta a(x,\theta) =y-x$.

\item
In the framework of Proposition~\ref{prop-5.3}(ii) for $Ch\le \ell(x,y) +|t|\le \epsilon$
\begin{gather}
a(x,\nabla_x \varphi )=\tau,\quad  \nabla_\theta \Phi (x,y,t,\theta) =0
\label{eqn-5.38}
\end{gather}
has has  no more than  a single solution $\theta$ and if it has then
$|t| \asymp \ell$ and $\nabla^2_\theta \Phi $ is a positive definite matrix.  This solution $\theta=\bar{\theta}+O(T)$ where 
$\bar{\theta}$ is a solution to $t\nabla_\theta a(x,\theta) =\tilde{y}-x$.
\end{enumerate}
\end{corollary}

\begin{proposition}\label{prop-5.9} 
In the framework of Proposition~\ref{prop-5.3}(ii) under extra assumption  
\begin{gather}
\sigma^2\ell \ge h^{1-\delta}
\label{eqn-5.39}\\
\intertext{the following estimate holds}
|F_{t\to h^{-1}\tau } \Bigl(\bar{\chi}_{T'}(t) u^1_h (x,y,t)- \bar{\chi}_{T}(t) U^1_h(x,y,t)\Bigr)|\le Ch^s
\label{eqn-5.40}
\end{gather}
with  $U^1_h(x,y,t)$ defined by \textup{(\ref{eqn-5.19})} with $B_n(x,y,t,\theta)$ described above and multiplied  by $\phi(\theta)$ supported  in $2\epsilon_3\rho $-vicinity and equal $1$ in $\epsilon_3\rho $-vicinity  of $\theta$ described in Corollary~\ref{corollary-4.8}(ii) and with $Ch\le T\le T' \le \epsilon$, $T=c \ell$, $\rho=\sigma^2$.
\end{proposition}

\begin{proof}
Let us fix $x=\bar{x}$ and $y=\bar{y}$  in (\ref{eqn-5.40}) and use $x$ as a variable. Due to Proposition~\ref{prop-5.1}  there is a single generalized billiard of the length 
$\le \epsilon$ with reflections at $\partial X$ from $\bar{x}$ to $\bar{y}$.  It has the length $\asymp \ell$, exactly one reflection and a reflection angle $\asymp \rho = (\eff(\bar{x})+\eff(\bar{y}))\ell^{-1}\gtrsim h^{\frac{1}{3}-\delta}$. 

Further, Proposition~\ref{prop-5.4}(ii) implies that  for $\sigma \ge h^{\frac{1}{3}-\delta}$ (which is due to (\ref{eqn-5.39}) and $\sigma \ge C_0\ell$) 
\textup{(\ref{eqn-5.19})} is a proper asymptotic series as $|t|\le T$ and we  define $U^1_h(x,y,t)$ through it (with cut-off). 

Furthermore, Corollary~\ref{corollary-4.8} implies that then $U^1_h (x,y,t)$ is negligible outside $\Omega_{2T,3\rho,3\sigma}$  
(by $(x_1,x',\xi')$) of this billiard  while the standard propagation results  imply that  $u^1_h(x,y,t) \,^t\!Q_y$ is also negligible outside of this vicinity provided
symbol of $Q$ is supported in the similar vicinity of $(\bar{y},\bar{\theta}')$, $\bar{\theta}'$ corresponds to this billiard.
Let  symbol of $Q$ be  $1$ in the similar vicinity of $(\bar{y},\bar{\xi}')$. Then  $ u^1_h(x,y,t) \,^t\!Q_y \equiv U^1_h (x,y,t)$ for $-T \le t\le  T$ 
with $0<T=C_0 \ell$.

On the other hand, in this case  $u^1_h(x,y,t) (I-^t\!Q_y)\equiv 0$ in $(\epsilon \rho \ell, \epsilon \ell )$-vicinity of $\bar{x}$ as $-T\le t\le \epsilon T$. 
Case of  $-\epsilon T \le t\le T$ is considered in the same way, albeit with the billiard from $\bar{x}$ to $\bar{y}$. Then we arrive to
(\ref{eqn-5.40}) with $T'=T$.

Finally, the standard propagation results imply that 
\begin{gather}
|F_{t\to h^{-1}\tau } \Bigl(\bigl(\bar{\chi}_{T'}(t) - \bar{\chi}_{T}(t)\bigr) u^1_h(x,y,t)\Bigr)|\le Ch^s
\label{eqn-5.41}
\end{gather}
as $T\le T'\le \epsilon$. Then (\ref{eqn-5.40}) expands to $T'\colon T\le T'\le \epsilon$. 
\end{proof}

\section{Spectral asymptotics}
\label{sect-5.2}

\begin{proposition}\label{prop-5.10}
In the framework of Proposition~\ref{prop-5.3}(ii) under extra assumption 
\begin{gather}
\sigma^2\ell \ge h^{1-\delta},\qquad \ell \ge h^{\frac{3}{5}}
\label{eqn-5.42}\\
\intertext{the following asymptotics holds}
e^{1,\T}_h(x,y,\tau) =(2\pi h)^{-d}\int _{a(y,\theta)<\tau } e^{ih^{-1}\varphi (x,y,\theta)}\,d\theta +O(h^{1-d}).
\label{eqn-5.43}
\end{gather} 
\end{proposition}

\begin{proof}
Let us   plug $U^1_h$ instead of $u^1_h$ into Tauberian expression $e^{1,\T}_h(x,y,\tau)$ with $T=\epsilon$.  
Let us decompose $B_n (x,y,t,\theta)$ into asymptotic series by $t^{k}$.

\begin{enumerate}[wide,label=(\alph*),  labelindent=0pt]
\item
Then terms with $k=0$ become
\begin{gather}
(2\pi h)^{-d} \int _{a(y,\theta)< \tau} \sum_{n\ge 0} e^{ih^{-1}\varphi (x,y,\theta)} B_n (x,y,0, \theta) h^n \,d\theta
\label{eqn-5.44}
\end{gather}
and we claim that these terms do not exceed $C\sigma^{-3n} h^{-d+n} (h\ell^{-1})^{(d+1)/2}$. Indeed, we know from Proposition~\ref{prop-5.3}(ii) that
$\varphi$ is uniformly infinitely smooth by $\theta$ and from Corollary~\ref{corollary-4.8} that there are no stationary points by $\theta$ and restriction of $\varphi (x,y,\theta)$ to $\Sigma (y,\tau)$ has two non-degenerate critical points. The leading terms would be of this magnitude, while all other terms would have extra factors not exceeding $h/\sigma^2\ell$ due to (\ref{eqn-5.22}). 

We want to estimate all extra terms by $Ch^{1-d}$. One can see easily, that $d=2$ is the worst case but even then therms with $n\ge 1$ are $O(h^{1-d})$ and terms with $h/\sigma^2\ell$ appear when we differentiate $B_0 (x,y,0,\theta)$ by $\theta$. Then due to (\ref{eqn-5.23}) these terms do not exceed
\begin{gather*}
C h^{-d} (h\ell^{-1})^{\frac{d+3}{2}} (1 + \ell \sigma^{-2} )
\end{gather*}
and this does not exceed $Ch^{1-d}$ under assumption (\ref{eqn-5.32}). \enlargethispage{\baselineskip}

Finally, since $B_0(x,y,0,\theta)= 1+O(\ell\sigma^{-1})$ we can replace $B_0(x,y,0,\theta)$ by $1$ resulting in the main part of asymptotics (\ref{eqn-5.43}).

\item
Consider now terms with $k\ge 1$ in the decomposition of $B_n(x,y,t,\theta)$. We can rewrite these terms as 
\begin{align}
&(2\pi h)^{1-d} \partial_\tau^{k-1}\int_{\Sigma  (y,\tau)} B'_{n,k} (x,y,\theta) h^{n+k-1} \,d\theta:d_\theta a(y,\theta)\notag\\
\intertext{and using $\xi$-microhyperbolicity rewrite it as}
&(2\pi h)^{1-d} \int_{\Sigma (y,\tau)} B''_{n,k} (x,y,\theta) h^{n} (h\ell^{-1})^k  \,d\theta:d_\theta a(y,\theta).
\label{eqn-5.45}
\end{align}
with $B''_{n,k}$  coming from $B_n$ with  $k$ derivatives by $t$ and no more than $(k-1)$ by $\theta$.  Then using (\ref{eqn-5.22})--(\ref{eqn-5.23}) we can estimate these terms by $Ch^{-\frac{d-1}{2}}\ell^{-\frac{d-1}{2}}\le Ch^{1-d}$.
\end{enumerate}
\end{proof}

\begin{remark}\label{remark-5.11}
The similar asymptotics 
\begin{gather}
e^{0,\T}_h(x,y,\tau) =(2\pi h)^{-d}\int _{a(y,\theta)<\tau } e^{ih^{-1}\varphi^0 (x,y,\theta)}\,d\theta +O(h^{1-d})
\label{eqn-5.46}
\end{gather} 
is well-known.
\end{remark}

\begin{remark}\label{remark-5.12}
Consider in the framework of Proposition~\ref{prop-5.3}(i), (ii)   integrals $I_h(\varphi )$ and $I_h(\varphi^0)$ in the right-hand expressions of (\ref{eqn-5.46}) and  (\ref{eqn-5.43}) correspondingly:
\begin{gather}
I_h (\varphi)\coloneqq (2\pi h)^{-d}\int _{a(y,\theta)<\tau } e^{ih^{-1}\varphi  (x,y,\theta)}\,d\theta.
\label{eqn-5.47}
\end{gather} 
\begin{enumerate}[wide,label=(\roman*),  labelindent=0pt]
\item
Let   $d\ge 3$. Then  $I_h (\varphi^0) = I _h (\bar{\varphi}^0)+O(h^{1-d})$ and  $I_h (\varphi) = I _h (\bar{\varphi})+O(h^{1-d})$
with  $\bar{\varphi}^0(x,y,\theta)= \langle x-y,\theta \rangle$, $\bar{\varphi}^0(x,y,\theta) =\langle x-\tilde{y},\theta \rangle$.
\item
Let $d=2$. Then  $I_h(\varphi^0) \equiv I_h(\bar{\varphi}^0)$ with an error $O(h^{-1})$ as $\ell ^0\ge h^{\frac{1}{3}}$, 
$O(h^{-1})$ as $\ell ^0(x,y)\ge h^{\frac{1}{3}}$, $O(h^{-\frac{1}{2}}\ell^{-\frac{3}{2}})$ as $h^{\frac{1}{2}}\le \ell^0(x,y)\le h^{\frac{1}{3}}$ and  $O(h^{-\frac{3}{2}}\ell^{\frac{1}{2}})$ as $h\le \ell ^0(x,y)\le h^{\frac{1}{2}}$. 

Also $I_h(\varphi) \equiv I_h(\bar{\varphi}^)$ with an error $O(h^{-1})$ as $\ell ^0\ge h^{\frac{1}{3}}$, $O(h^{-1})$ as $\ell ^0(x,y)\ge h^{\frac{1}{3}}$, $O(h^{-\frac{1}{2}}\ell^{-\frac{3}{2}})$ as $h^{\frac{1}{2}}\le \ell (x,y)\le h^{\frac{1}{3}}$ and  $O(h^{-\frac{3}{2}}\ell^{\frac{1}{2}})$ as $h\le \ell (x,y)\le h^{\frac{1}{2}}$. 
\end{enumerate}
\end{remark}

\begin{proof}
Observe first that due to the stationary phase principle $I_h (\varphi) = Ch^{-d}(h\ell^{ -1})^{(d+1)/2}$  and the same is true for $I_h (\bar{\varphi})$.
Therefore for $\ell \ge h^{\frac{1}{2}}$ we have simply estimates rather than the error estimates. 

On the other hand due to proposition~\ref{prop-5.7}  $\varphi (x,y,\theta)= \langle x-\tilde{y},\theta \rangle + O(\ell^2)$, for $\ell \le h^{\frac{1}{2}}$ the error does not exceed
$Ch^{-d}(h\ell^{ -1})^{(d+1)/2}\times \ell^2h^{-1}$ which is $O(h^{1-d})$.

The same is true for $I_h (\varphi^0)$ and $I_h (\bar{\varphi}^0)$ albeith with $\ell^0$ rather than $\ell$.
\end{proof}

Now we need to deal with $d=2$ and $\ell^0\le h^{\frac{1}{3}}$ or  $\ell \le h^{\frac{1}{3}}$.

\begin{proposition}\label{prop-5.13}
Let $d=2$. Then
\begin{enumerate}[wide,label=(\roman*),  labelindent=0pt]
\item
In the framework of Proposition~\ref{prop-5.3}(i) for $h^{1-\delta}\le \ell^0(x,y)\le h^{\frac{1}{3}}$
\begin{gather}
\int_{a(y,\theta)<\tau} e^{ih^{-1}\varphi^0(x,y,\theta)}\,d\theta = \int_{a(w,\xi )<\tau} e^{ih^{-1}\langle x-y,\xi\rangle }\,d\xi +O(h)
\label{eqn-5.48}
\end{gather}
with  $w=\frac{1}{2}(x+y)$.
\item
In the framework of Proposition~\ref{prop-5.3}(ii) for $h^{1-\delta}\le \ell (x,y)\le h^{\frac{1}{3}}$
\begin{multline}
\int_{a(y,\theta)<\tau} e^{ih^{-1}\varphi (x,y,\theta)}\,d\theta \\
= \int_{a(w,\xi )<\tau} e^{ih^{-1}(\langle x-\tilde{y},\xi\rangle + \frac{1}{2}\varkappa (w',\xi',\tau)(x_1^2+y_1^2)) }\,d\xi +O(h)
\label{eqn-5.49}
\end{multline}
with $w=(0,\frac{1}{2}(x'+y'))$ and with 
\begin{gather}
\varkappa (x',\xi',\tau)=(\tau-b(x_2,\xi_2))^{-\frac{1}{2}}a_{x_1}(x,\xi)|_{x_1=\xi_1=0};
\label{eqn-5.50}
\end{gather}
recall that $b(w_2,\xi_2)=a(x,\xi)|_{x_1=\xi_1=0}$.
\end{enumerate}
\end{proposition}

\begin{proof}
\begin{enumerate}[wide,label=(\roman*),  labelindent=0pt]
\item
Consider $(z,\bar{\xi})\colon a(\bar{w},\bar{\xi})=\tau$ and Hamiltonian trajectory $\Psi_t (\bar{w},\bar{\xi})$ of 
$a(x,\xi)$, $\Psi_0 (\bar{w},\bar{\xi})=(\bar{w},\bar{\xi})$. Then as  $(y(t), \theta(t))=\Psi_t(\bar{w},\bar{\xi})$ and $(x(t),.)=\Psi_{-t}(\bar{w},\bar{\xi})$, $|t|\le C_0\ell$  we have
\begin{gather}
\varphi^0 (x(t),y(t),\theta(t))= \langle x(t)-y(t), \bar{\xi} \rangle + O(\ell^3)
\label{eqn-5.51}
\end{gather}
where  in this part of the proof for simplicity we write $\ell$ rather than $\ell^0$.

Observe that as $\ell \ge h^{1-\delta} $ we can reduce  integrals in (\ref{eqn-5.48}) to $\Sigma (y,\tau)$ and $\Sigma (\bar{w},\tau)$ respectively, gaining factor $h\ell^{-1}$. Consider first those reduced integrals over $\epsilon$-vicinities of $\theta(t)$ and $\bar{\xi}$. We see that there is a diffeomorphism of these vicinities and
\begin{gather*}
\varphi^0(x(t),y(t), \theta (\bar{w}, \xi(s), t)) =  \langle x(t)-y(t),  \xi  \rangle +  O(\ell^3+s^2\ell )
\end{gather*}
provided $a_\theta \parallel \theta$ at $(\bar{w},\bar{\xi})$ (we can assume it without any loss of the generality), $s$ is an angle parameter on $\Sigma(\bar{w},\tau)$, $\xi(0)=\bar{\xi}$. Making change of variables we estimate  an error by $C(h\ell^{-1})^{\frac{3}{2}}\times (\ell^3h^{-1}+\varepsilon^2 \ell h^{-1} + \ell)$ when we integrate over $|s|\le \varepsilon$, and when we integrate over $|s|\ge \varepsilon$, we can multiply it by an arbitrary large power of $h/\varepsilon ^2\ell$. Taking $\varepsilon =(h\ell^{-1})^{\frac{1}{2}}$ we make both errors  $O(h)$.

On the other hand, fix $\bar{x}$ and $\bar{y}$ and find corresponding $\bar{w}$, $\bar{\xi}$ and $t$. One can see easily that $\bar{w}=w+O(\ell^2)$ and the error in
the integral in the right hand expression of (\ref{eqn-5.48}) when we redefine $w$ does not exceed $C(h\ell)^{\frac{1}{2}}\ell^2 =O(h)$.  

\item
Making change of variables $\x_1= x_1+\frac{1}{2}\varkappa (w_2,\xi',\tau)x_1^2$ and therefore  $\y_1= y_1+\frac{1}{2}\varkappa (w_2,\xi',\tau)y_1^2$ we would arrive to operator with the symbol which, after division by $(a-\tau)$, modulo  $O(\ell^2)$ does not depend on $\x_1$  and therefore with the corresponding phase function 
\begin{gather}
\varphi   = (\x_1+\y_1)\xi_1 + (x_2-y_2)\xi_2 + O(\ell^3).
\notag
\intertext{Then we can use the the method of reflection and after this, Statement~(i). So we get}
\int_{a(w,\xi )<\tau} e^{ih^{-1}(  (\x_1+\y_1)\xi_1  + (x_2-y_2)\xi_2}\,d\xi +O(h)
\label{eqn-5.52}
\end{gather}
with $w =(0, \frac{1}{2}(x_2+y_2))$; which is exactly the right-hand expression of (\ref{eqn-5.49}). 
\end{enumerate}
\end{proof}

\begin{remark}\label{remark-5.14}
 \begin{enumerate}[wide,label=(\roman*),  labelindent=0pt]
 \item
 Therefore in Theorem~\ref{thm-1.2} 
\begin{multline}
e_{h,\corr} (x,y,\tau)\\
 =[\mp] (2\pi h)^{-1}\int _{\{a(\frac{1}{2}(x+y),\xi)<\tau\}} e^{ih^{-1}\langle x-\tilde{y},\xi\rangle}
\Bigl(e^{ih^{-1} \frac{1}{2}\varkappa (\frac{1}{2}(x_2+y_2),\xi_2,\tau)}-1\Bigr)\,d\xi
\label{eqn-5.53}
\end{multline}
with $\varkappa (x_2,\xi_2,\tau)$ defined by (\ref{eqn-5.50}).
\item
One can see easily that in the framework of Proposition~\ref{prop-4.5} expressions (\ref{eqn-5.53}) and (\ref{eqn-4.49}) coincide modulo $O(h^{-1})$.
\end{enumerate}
\end{remark}

\chapter{Synthesis and final remarks}
\label{sect-6}
\begin{proof}[Proof of Theorems~\ref{thm-1.1} and~\ref{thm-1.2}]
We know that for $\ell(x,y)\le \epsilon$ in the frameworks of Theorems~\ref{thm-1.1} and\ref{thm-1.2} estimate (\ref{eqn-1.8}) holds.
 \begin{gather*}
 e_h(x,y,\tau) = e^\T _{T,h} (x,y,\tau)+O(h^{1-d}).
 \end{gather*}
  \begin{enumerate}[wide,label=(\roman*),  labelindent=0pt]
 \item
We also know from Section~\ref{sect-2} that in the framework of Theorem~\ref{thm-1.1} $e^\T _{T,h} (x,y,\tau)$ can be defined as an oscillatory integral and Proposition~\ref{prop-5.13}(i) implies that modulo $O(h^{1-d})$ this oscillatory integral can be rewritten as $e^\W_h(x,y,\tau)$ defined by (\ref{eqn-1.3}). This proves 
Theorem~\ref{thm-1.1}.

\item
In the framework of Theorem~\ref{thm-1.2} we established in Section~\ref{sect-3}  
\begin{gather*}
e^{1,\T} _{T,h} (x,y,\tau)=O(h^{1-d})
\end{gather*}
 as $d\ge 3$ and $\ell(x,y)\ge h^{\frac{1}{2}-\delta}$ and as $d=2$ and  $ \ell(x,y)\ge h^{\frac{1}{3}-\delta}$. 
Further, we proved there that $e^{1,\T} _{T,h} (x,y,\tau)=O(h^{1-d-\delta})$ as $d=2$ and $\ell (x,y)\ge h^{\frac{1}{3}+\delta}$. In both cases we also assume there that $|x'-y'| \ge \epsilon \ell(x,y)$. However it follows from Section~\ref{sect-5} that this latter condition is not necessary. Combined with Remark~\ref{remark-1.3}(ii) it proves Theorem~\ref{thm-1.2} in this case.

Let $d\ge 3$. 
In Section~\ref{sect-4}  we proved   asymptotics  
\begin{gather}
e^{1,\T}_{T,h}(x,y)= [\mp] e^{0,\W}_h(x,\tilde{y},\tau)+O(h^{1-d})
\label{eqn-6.1}
\end{gather}
if either  $\ell (x,y)\le h^{\frac{1}{2} +\delta}$ \underline{or}  $ h^{\frac{1}{2}+\delta}\le \ell(x,y)  \le h^{\frac{1}{2} -\delta}$ and 
\begin{gather}
\eff(x)+\eff(y)\le \sigma \ell (x,y)
\label{eqn-6.2}
\end{gather}
with $ \sigma = h^{2\delta}$. In Section~\ref{sect-5}  we proved the same asymptotics as 
 \begin{gather}
\eff(x)+\eff(y)\ge \sigma \ell (x,y)
\label{eqn-6.3}
\end{gather}
with $\sigma= h^{\frac{1}{2}-\delta}\ell^{-\frac{1}{2}}+C_0 \ell$. These domains overlap and cover all values of $\ell \le h^{\frac{1}{2}-\delta}$  and $\sigma $.

\item
Let $d=2$. Then instead of asymptotics  (\ref{eqn-6.1})  we have
\begin{gather}
e^{1,\T}_{T,h}(x,y)= [\mp] e^{0,\W}_h(x,\tilde{y},\tau)+e_{h,\corr}(x,y,\tau)+ O(h^{1-d})
\label{eqn-6.4}
\end{gather}
 and Sections~\ref{sect-4} and~\ref{sect-5} cover cases (\ref{eqn-6.2}) with $\sigma =h^{1+\delta}\ell^{-2}$ and (\ref{eqn-6.3}) with
 $\sigma =h^{\frac{1}{3}-\delta}$, $\ell(x,y)\le h^{\frac{1}{3}+\delta}$.  These domains overlap and cover all values  $\ell \le h^{\frac{1}{3}+\delta}$ and $\sigma$.
Here we are left with domain $h^{\frac{1}{3}+\delta }\le \ell\le h^{\frac{1}{3}-\delta}$, $\sigma =h^{\frac{1}{3}-\delta}$ where we have less precise estimate
$e^{1,\T} _{T,h} (x,y,\tau)=O(h^{1-d-\delta})$.
\end{enumerate}
\end{proof}
\enlargethispage{\baselineskip}

\begin{remark}\label{remark-6.1}
\begin{enumerate}[wide,label=(\roman*),  labelindent=0pt]
\item
We would like to derive remainder estimate $O(h^{-1})$ as $d=2$ and $\ell(x,y)\le \epsilon$, thus removing the gap  $h^{\frac{1}{3}+\delta} \le \ell \le h^{\frac{1}{3}-\delta}$, $\eff (x)+\eff(y)\le h^{\frac{1}{3}-\delta}$.

\item
For Schr\"odinger operator (\ref{eqn-1.4})  can get rid of $\xi$-microhyperbolicity assumption by rescaling method  using scaling function 
\begin{gather}
\gamma_x=(\epsilon |V(x)-\tau|+h^{\frac{2}{3}} ).
\label{eqn-6.5}
\end{gather}
Then using   arguments of Proposition~\ref{OOD2-prop-3.2} and Remark~\ref{OOD2-remark-4.1} of \cite{OOD2} we have a remainder estimate $O(h^{1-d}\gamma_x^{(d-3)/2}\gamma_y^{(d-3)/2})$  which as $d\ge 3$ is $O(h^{1-d})$; however   as $d=2$ it is as bad as $O(h^{-\frac{4}{3}})$.

\item
As $d=2$   away from $\partial X$ remainder estimate $O(h^{-1})$  in the regular zone and $O(h^{-\frac{16}{15}})$ or better in the singular zone has been derived in \cite{OOD2}. However it required the analysis of the Hamiltonian trajectories and introduction of the correction term. This does not look feasible near the boundary
especially because both  the boundary and degeneration define correction term.
\end{enumerate}
\end{remark}

\bibliographystyle{amsplain}

\end{document}